\documentclass[11pt]{article}

\usepackage{tikz,amssymb}
\usepackage{subfig}
\usepackage{hyperref}
\def\thetitle{{Coxeter groups, hyperbolic cubes, and acute triangulations}}
\hypersetup{
    pdftitle=   \thetitle,
   pdfauthor=  {Sang-hyun Kim and Genevieve S. Walsh}
}
\usepackage{enumerate}
\makeatletter
\let\@@enum@org\@@enum@
\def\@@enum@[#1]{\@@enum@org[\normalfont #1]}
\makeatother

\usepackage{graphicx}
\usepackage{amsmath}
\usepackage{amsthm}
\usepackage{latexsym} 
\usepackage{color}

\newcommand\form[1]{\langle #1\rangle}
\newtheorem{theorem}{Theorem}[section]
\newtheorem{lemma}[theorem]{Lemma}
\newtheorem{corollary}[theorem]{Corollary}
\newtheorem{conjecture}[theorem]{Conjecture}
\newtheorem{proposition}[theorem]{Proposition}
\newtheorem{question}[theorem]{Question}

\newtheorem{observation}[theorem]{Observation} 

\theoremstyle{definition}
\newtheorem{definition}[theorem]{Definition}
\newtheorem{example}[theorem]{Example}
\newtheorem{remark}[theorem]{Remark}

\newcommand\co{\colon}

\newcommand\EE{\mathbb{E}^2}

\newcommand\HH{\mathbb{H}^2}

\newcommand\HHH{\mathbb{H}^3}
\newcommand\ZZ{\mathbb{Z}/2\mathbb{Z}}

\newcommand\Z{\mathbb{Z}}
\newcommand\CAT{\mathrm{CAT}}
\newcommand{\area}{\ensuremath \mathrm{Area}}
\newcommand\link{\operatorname{Link}}

\title{Coxeter groups, hyperbolic cubes, and acute triangulations} 

\author{S. Kim and G. S. Walsh} 
\begin{document} \maketitle

\abstract{Let $C(L)$ be the right-angled Coxeter group defined by an abstract triangulation $L$ of $\mathbb{S}^2$.  We show that $C(L)$ is isomorphic to a hyperbolic right-angled reflection group if and only if $L$ can be realized as an acute triangulation.  The proof relies on the theory of $\CAT(-1)$ spaces. A corollary is that an abstract triangulation of $\mathbb{S}^2$ can be realized as an acute triangulation exactly when it satisfies a combinatorial condition called ``flag no-square''. We also study generalizations of this result to other angle bounds, other planar surfaces and other dimensions.}

\section{Introduction}\label{s:intro}
Here we study Coxeter groups and a \emph{geometric} condition on their defining graphs which is equivalent to these groups being one-ended and word-hyperbolic.  

For a simplicial complex $L$ the \emph{right-angled Coxeter group $C(L)$ defined by $L$} is the group with presentation: 
\[\form{v\in L^{(0)} \big| a^2 =1\text{ for }a\in L^{(0)},\;\; [a,b] = 1 \text{ for } \{a,b\} \in L^{(1)}}.\]
See~\cite{Davisbook} for extensive collection of results on these groups and on general Coxeter groups.

An \emph{abstract triangulation} of the unit two-sphere $\mathbb{S}^2$ means a simplicial complex homeomorphic to $\mathbb{S}^2$. 
An \emph{acute triangulation} of $\mathbb{S}^2$ is a triangulation of $\mathbb{S}^2$ into geodesic triangles whose dihedral angles are all acute.
We say an abstract triangulation $L$ of $\mathbb{S}^2$ is \emph{realized by an acute triangulation of $\mathbb{S}^2$}  (or in short, \emph{acute}) if there is an acute triangulation $T$ of $\mathbb{S}^2$ and a simplicial homeomorphism from $T$ to $L$.

The main question of this paper is the following.

\begin{question}\label{que:main}
When is an abstract triangulation of $\mathbb{S}^2$ acute?
\end{question}

The boundary of an icosahedron is an example of an abstract triangulation of $\mathbb{S}^2$ that is acute.
Indeed, an equilateral spherical triangle of dihedral angle $2\pi/5$ gives the icosahedral tessellation of $\mathbb{S}^2$.
If a triangulation of $\mathbb{S}^2$ by geodesic triangles has a vertex whose degree is three, then at least one of the three dihedral angles at that vertex is not acute, since the angles must add up to $2 \pi$.  Hence the boundary of a tetrahedron is an easy non-example.
% since the degree of each vertex is three.
The question is more subtle in many other cases. 
\begin{example} 
Consider the abstract triangulations of $\mathbb{S}^2$ obtained by doubling the squares in Figure~\ref{fig:oum} along their boundaries. Figure~\ref{fig:oum} (a) is due to Oum (private communication).  Every vertex in these triangulations of $\mathbb{S}^2$ has degree strictly greater than 4.  Furthermore, Figure~\ref{fig:oum} (a) admits an acute angle structure.  By this we mean that it admits an assignment of acute angles such that the sum of the angles around each vertex is equal to $2\pi$ and the sum of the angles in each triangle is greater than $\pi$.  However, Corollary \ref{cor:main} below implies that neither can be realized as an acute triangulation. \end{example}

For $\delta\ge0$ we say a geodesic metric space $X$ is \emph{$\delta$-hyperbolic} if every geodesic triangle in $X$ has the property that each side is contained in the $\delta$-neighborhood of the other two sides; see~\cite{BH1999} for details. Suppose $G$ is a group generated by a finite set $S$ and $\Gamma$ is the corresponding Cayley graph. 
Put a length metric on $\Gamma$ by declaring that each edge has length one. The group $G$ is \emph{word-hyperbolic} if this metric graph $\Gamma$ is $\delta$-hyperbolic for some $\delta\ge0$.
We say $G$ is \emph{one-ended} if $\Gamma$ has the property that the complement of every finite collection of edges has exactly one infinite component; this is equivalent to saying that $G$ is infinite and does not split over a finite group~\cite{Stallings1968}. Our main theorem is as follows.

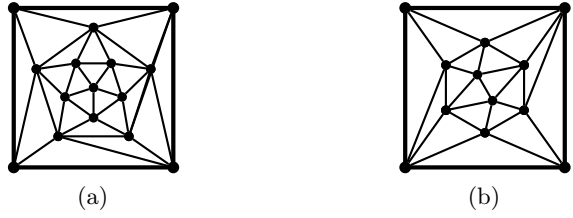
\begin{figure}[t]
\begin{center}
\subfloat[(a)]{
  \tikzstyle {a}=[red,postaction=decorate,decoration={%
    markings,%
    mark=at position .5 with {\arrow[red]{stealth};}}]
  \tikzstyle {b}=[blue,postaction=decorate,decoration={%
    markings,%
    mark=at position .43 with {\arrow[blue]{stealth};},%
    mark=at position .57 with {\arrow[blue]{stealth};}}]
  \tikzstyle {v}=[draw,circle,fill=black,inner sep=1pt]
   \tikzstyle {w}=[draw,circle,fill=white,inner sep=1pt]
\begin{tikzpicture}[thick,scale=0.5]
\node [v] at (0:0) (z) {};
\foreach \i in {0,...,4}{ % Pentagons
    	\draw (360/5*\i+90:1.6) node (pp\i) [v] {}--(360/5*\i+360/5+90:1.6) node (p\i) [v] {};
    	\draw (360/5*\i+90-36:.8)--(360/5*\i+360/5+90-36:.8) node (q\i) [v] {};
	\draw (q\i)--(z);
	\draw (pp\i)--(q\i)--(p\i);
}
\foreach \i in {0,...,3}{ % Big Square
    	\draw [ultra thick] (360/4*\i+45:3)--(360/4*\i+360/4+45:3) node (x\i) [v] {};
	\draw (x\i)--(p\i);
	\draw (x\i)--(pp\i);
}
\draw (x2)--(p3);
\draw (x3)--(p4);
\node  [] at (3,0) {};  % hidden point to match 2 figures
\end{tikzpicture}
}
$\qquad\qquad$
\subfloat[(b)]{
  \tikzstyle {a}=[red,postaction=decorate,decoration={%
    markings,%
    mark=at position .5 with {\arrow[red]{stealth};}}]
  \tikzstyle {b}=[blue,postaction=decorate,decoration={%
    markings,%
    mark=at position .43 with {\arrow[blue]{stealth};},%
    mark=at position .57 with {\arrow[blue]{stealth};}}]
  \tikzstyle {v}=[draw,circle,fill=black,inner sep=1pt]
   \tikzstyle {w}=[draw,circle,fill=white,inner sep=1pt]
\begin{tikzpicture}[thick,scale=0.5]
\draw (120:.4) node (q0) [v] {}--(300:.4) node (q1) [v] {};
\foreach \i in {0,1,4,5}{ % Pentagons
    	\draw (360/6*\i+360/6+90:1.2) node (p\i) [v] {}--(q0);
}
\foreach \i in {1,2,3,4}{ % Pentagons
    	\draw (360/6*\i+360/6+90:1.2) node (p\i) [v] {}--(q1);
}
\foreach \i in {0,...,5}{ % Pentagons
    	\draw (p\i)--(360/6*\i+2*360/6+90:1.2) node (pp\i) [v] {};
}
\foreach \i in {0,...,3}{ % Big Square
    	\draw [ultra thick] (360/4*\i+45:3)--(360/4*\i+360/4+45:3) node (x\i) [v] {};
}
\draw (x1)--(p1);
\draw (x3)--(p4);
\draw (x0)--(p0)--(x1)--(p2)--(x2)--(p3)--(x3)--(p5)--(x0);
\node  [] at (3,0) {};  % hidden point to match 2 figures
\end{tikzpicture}
}
\end{center}
\caption{Triangulations of squares.}\label{fig:oum}
\end{figure}

%Use CAT(-1) to be consistent with the beginning of the paper
\begin{theorem} \label{th:main racg} 
An abstract triangulation $L$ of $\mathbb{S}^2$ is acute if and only if $C(L)$ is one-ended and word-hyperbolic.
\end{theorem}

If $L$ is a triangulation of $\mathbb{S}^2$, then $C(L)$ is the orbifold-fundamental group of a reflection orbifold $\mathcal{O}$.  If furthermore $C(L)$ is one-ended and word-hyperbolic, $\mathcal{O}$ is hyperbolic by \cite{BLP2005}; see also \cite{orbifoldbook}.  
Thus we have that an abstract triangulation $L$ of $\mathbb{S}^2$ can be realized as an acute triangulation $T$ exactly when the associated reflection orbifold of $C(L)$ can be realized as a hyperbolic orbifold $\mathcal{O}$. 
%Given an abstract triangulation which can be realized as an acute triangulation there is a whole space of such acute triangulations, and one (up to M\"obius transformations) which is geometric, see Section \ref{s:further}.  
Among the infinitely many acute triangulations $T$ realizing such an $L$, exactly one triangulation $T_0$ (up to M\"obius transformations) will be \emph{geometric}, meaning that $T_0$ corresponds to the `` dual projection'' of $\mathcal{O}$ onto $\partial\HHH$; see Section \ref{s:further} for details.

A simplicial complex $Y$ is \emph{flag} if every complete subgraph in $Y^{(1)}$ spans a simplex.
Following~\cite{JS2003}, we say $Y$ is \emph{no-square} if every 4-cycle in $Y^{(1)}$ has a chord in $Y^{(1)}$.
So an abstract triangulation $L$ of $\mathbb{S}^2$ is flag no-square if and only if there does not exist a 3- or 4-cycle $C$ such that each component of $L\setminus C$ contains a vertex.
For an abstract triangulation $L$ of $\mathbb{S}^2$, 
we have that $C(L)$ is one-ended and word-hyperbolic if and only if $L$ is flag no-square (Lemma~\ref{lem:sep}).
Therefore, we have a combinatorial answer to Question~\ref{que:main} as follows.

\begin{corollary} \label{cor:main} 
An abstract triangulation $L$ of $\mathbb{S}^2$ is acute if and only if $L$ is flag no-square.
\end{corollary} 
 
By combining Corollary~\ref{cor:main} with a result of Itoh (\cite{Itoh2001} and Theorem~\ref{thm:itoh}), we will have the following characterization of the number of faces; see Section~\ref{s:history} for details.

\begin{corollary} \label{cor:faces} 
There exists an acute triangulation of $\mathbb{S}^2$ with $n$ faces if and only if $n$ is even, $n\ge20$ and $n\ne22$.
\end{corollary}

Although Corollary \ref{cor:main} and \ref{cor:faces} are stated in terms of elementary spherical geometry,
we find it more natural to apply the tools of geometric group theory and hyperbolic geometry for the proof as exhibited in this paper. We occasionally do use spherical geometry, for example in the proof of  Lemma~\ref{lem:empty}.
Using similar but more sophisticated techniques, we will also extend Theorem~\ref{th:main racg} to other Coxeter groups (Theorems \ref{thm:main general} and \ref{thm:main 22p}).

An abstract triangulation $L$ of a planar surface is \emph{acute} in $\mathbb{S}^2$ or in $\EE$ if there exists a simplicial homeomorphism from $L$ onto an acute geodesic triangulation of a subset of $\mathbb{S}^2$ or of $\EE$.
We have a generalization of Corollary~\ref{cor:main} to planar surfaces as follows;
exact definitions will be given in Sections~\ref{s:history} and~\ref{s:planar}.
We remark that Maehara proved a result similar to (2) when $L\approx\mathbb{D}^2$; see~\cite{Maehara2003} and Theorem~\ref{thm:maehara}.  

\begin{theorem}\label{th:planar}
Let $L$ be an abstract triangulation of a compact planar surface
such that $L$ is flag no-separating-square. Then: 
\begin{enumerate}[(1)]
\item
$L$ is acute in $\mathbb{S}^2$.
\item 
$L$ is acute in $\mathbb{E}^2$ if and only if at least one boundary component of $L$ is not a square.
\end{enumerate}
%In both of the cases, we can realize $L$ as an acute triangulation with coinciding perpendiculars.
\end{theorem}

In particular, the triangulations of $\mathbb{D}^2$ given in Figure~\ref{fig:oum} are acute in $\mathbb{S}^2$. However, these triangulations are not acute in $\EE$ by part (2) of the theorem or by~\cite{Maehara2003}.
We will lastly address the acute triangulability for higher dimensional spheres, using the main theorem of~\cite{Kalai1990}.
Acute triangulations of higher-dimensional spaces mean triangulations into simplices with acute dihedral angles.

\begin{proposition}\label{prop:higher}
The $d$--dimensional sphere admits an acute spherical triangulation if and only if $d\le3$.
\end{proposition}

We hope that combinatorists will find Corollary \ref{cor:main} useful in their work.  
Sleator, Tarjan and Thurston utilized hyperbolic geometry in order to prove that every pair of abstract triangulations of a disk with $n$ vertices can be connected by at most $O(n)$ \emph{diagonal flips}, namely changing the diagonal of a square formed by two faces~\cite{STT1988}.
Komuro proved a similar result for $\mathbb{S}^2$~\cite{Komuro1997}. 
From Corollary~\ref{cor:main}, we observe that the property of being realizable as an acute triangulation can be destroyed by one diagonal flip.
For example, the triangulation of $\mathbb{S}^2$ given by two copies of Figure~\ref{fig:oum} can be flipped by changing one of the four edges on the square. The resulting triangulation is flag no-square, and hence acute in $\mathbb{S}^2$.
We also hope that topologists will have interest in our proof of $\CAT(1)$-ness of certain spherical complexes homeomorphic to $\mathbb{S}^2$ (Section~\ref{s:dual}).

Section~\ref{s:history} will give some background on acute triangulations and the proof of Corollary~\ref{cor:faces}. Section~\ref{s:prelim} will review relevant background on $\CAT(\kappa)$ spaces and hyperbolic polyhedra.
Section~\ref{s:key} contains necessary results regarding reflection cubes and dual Davis complexes.
We prove Theorem~\ref{th:main racg} in Section~\ref{s:main}.  One direction follows from the fact that $C(L)$ is a right-angled hyperbolic orbifold group. For the other direction we consider strongly $\CAT(1)$ structures on $\mathbb{S}^2$ to prove that $C(L)$ is %$\CAT(-1)$. 
word-hyperbolic.
Sections~\ref{s:coxeter} through~\ref{s:subordinate} generalize Theorem~\ref{th:main racg} to other angle bounds. For this purpose we will develop a more refined argument for proving certain 2-complexes are $\CAT(1)$ in Section~\ref{s:dual}.
%These structures come from hyperbolic polyhedra~\cite{Hodgsonfirstauthor} and are described in Section~\ref{s:prelim}.
Section \ref{s:planar} discusses acute triangulations of planar surfaces.  
In Section \ref{s:further}, we ask further questions about spaces of acute triangulations
and address the situation for higher-dimensional spheres. 
%\comment{revise this summary after revision}

\section{Acute triangulations} \label{s:history}
The three main themes surrounding the topic of acute triangulations are the following: triangulability of spaces, number of simplices and triangulability with fixed combinatorics.
 
\iffalse
By an \emph{abstract triangulation of a manifold $M$}, we mean a simplicial complex homeomorphic to $M$.  If $M$ is equipped with a Riemmanian metric, a \emph{geodesic triangulation} of $M$ is a triangulation of $M$ into totally geodesic simplices.  By an \emph{acute triangulation} of $M$, we mean a geodesic triangulation such that each simplex has acute dihedral angles.   An abstract triangulation $L$ is \emph{realized by an acute triangulation in $M$}  (in short, \emph{acute} in $M$) if there is an acute triangulation $T$ of a subset of $M$ and a simplicial homeomorphism from $T$ to $L$.  
\fi

The question of acute triangulability of $\mathbb{E}^3$ can be traced back to Aristotle, as pointed out in~\cite{BKKS2009}. The (false) claim of Aristotle that the regular tetrahedra tessellate $\mathbb{E}^3$ was not refuted until the Middle Ages. 
An acute triangulation of $\mathbb{E}^3$ was first constructed in~\cite{ESU2004},
and that of 3--dimensional cubes was discovered later~\cite{KPP2012,VHZG2010}.
An acute triangulation of $\mathbb{E}^n$ does not exist for $n\ge5$~\cite{Kalai1990,KPP2012}.
There is an acute triangulation of an $n$-cube exactly when $n \leq3$~\cite{KPP2012}.
At the time of this writing,
it is not known whether $\mathbb{E}^4$ has an acute triangulation~\cite{KPP2012}.
Colin de Verdi\'ere and Marin proved that every closed surface equipped with a Riemannian metric admits an \emph{almost equilateral geodesic triangulation}, which in particular implies that each triangle is acute~\cite{MV1990}.
The acute triangulability question is related to topics in numerical analysis such as
piecewise polynomial approximation theory and finite element method for PDE; see~\cite{BKKS2009} for further reference.

Regarding the number of simplices, Gardner first asked how many acute triangles are needed to triangulate an obtuse triangle~\cite{Gardner1995}. This was answered to be seven~\cite{GM1960}. More generally, an arbitrary $n$-gon in $\mathbb{E}^2$ can be triangulated into $O(n)$ acute triangles~\cite{Maehara2002}. The number of simplices used for an acute triangulation of the unit cube $[0,1]^3$ was $2715$ in the construction of~\cite{KPP2012} and $1370$ in~\cite{VHGR2009}.
A geodesic triangle contained in one hemisphere of $\mathbb{S}^2$ can be triangulated into at most ten acute triangles and this bound is sharp~\cite{IZ2002}.
For triangulations of $\mathbb{S}^2$, Itoh proved the following theorem mainly by explicit constructions.

\begin{theorem}[\cite{Itoh2001}]\label{thm:itoh}
\begin{enumerate}[(1)]
\item
If there exists an acute triangulation of $\mathbb{S}^2$ with $n$ faces, 
then $n$ is even, $n\ge20$ and $n\ne 22$.
\item
If $n$ is even, $n\ge20$ and $n\ne 22,28,34$, then there exists an acute triangulation of $\mathbb{S}^2$ with $n$ faces.
\end{enumerate}
\end{theorem}

Itoh then asked whether or not there exists an acute triangulation of $\mathbb{S}^2$ with either 28 or 34 faces.
Using CaGe software~\cite{BDLPV2010}, we exhibit examples of flag no-square triangulations with 28 and 34 faces
 in Figure~\ref{fig:face}; see also~\cite{BM2005}. 
Since we will prove that a flag no-square triangulation can be realized as an acute triangulation (Corollary~\ref{cor:main}),  we have a complete characterization of the number of faces in an acute triangulation of $\mathbb{S}^2$ as given in Corollary~\ref{cor:faces}.

\begin{figure}[htb]
\begin{center}
\subfloat[(a)]{
\begin{tikzpicture}[scale=0.04]
  \tikzstyle {v}=[draw,circle,fill=black,inner sep=.7pt]
    \definecolor{marked}{rgb}{0.25,0.5,0.25}
    \node [v,fill] (16) at (59.168964,22.885309) {};
    \node [v,fill] (15) at (40.831037,22.885309) {};
    \node [v,fill] (14) at (49.989400,35.499258) {};
    \node [v,fill] (13) at (65.571338,33.951664) {};
    \node [v,fill] (12) at (78.672886,19.005725) {};
    \node [v,fill] (11) at (49.947001,11.564553) {};
    \node [v,fill] (10) at (21.305914,18.984524) {};
    \node [v,fill] (9) at (34.449863,33.930464) {};
    \node [v,fill] (8) at (43.608226,49.745601) {};
    \node [v,fill] (7) at (56.391774,49.745601) {};
    \node [v,fill] (6) at (70.701717,47.689210) {};
    \node [v,fill] (5) at (99.999999,6.688575) {};
    \node [v,fill] (4) at (0.000000,6.688575) {};
    \node [v,fill] (3) at (29.298283,47.562009) {};
    \node [v,fill] (2) at (49.925801,68.528726) {};
    \node [v,fill] (1) at (49.989400,93.311424) {};
    \draw [black] (16) to (11);
    \draw [black] (16) to (15);
    \draw [black] (16) to (14);
    \draw [black] (16) to (13);
    \draw [black] (16) to (12);
    \draw [black] (15) to (9);
    \draw [black] (15) to (14);
    \draw [black] (15) to (11);
    \draw [black] (15) to (10);
    \draw [black] (14) to (7);
    \draw [black] (14) to (13);
    \draw [black] (14) to (9);
    \draw [black] (14) to (8);
    \draw [black] (13) to (6);
    \draw [black] (13) to (12);
    \draw [black] (13) to (7);
    \draw [black] (12) to (5);
    \draw [black] (12) to (11);
    \draw [black] (12) to (6);
    \draw [black] (11) to (4);
    \draw [black] (11) to (10);
    \draw [black] (11) to (5);
    \draw [black] (10) to (3);
    \draw [black] (10) to (9);
    \draw [black] (10) to (4);
    \draw [black] (9) to (3);
    \draw [black] (9) to (8);
    \draw [black] (8) to (2);
    \draw [black] (8) to (7);
    \draw [black] (8) to (3);
    \draw [black] (7) to (2);
    \draw [black] (7) to (6);
    \draw [black] (6) to (1);
    \draw [black] (6) to (5);
    \draw [black] (6) to (2);
    \draw [black] (5) to (1);
    \draw [black] (5) to (4);
    \draw [black] (4) to (1);
    \draw [black] (4) to (3);
    \draw [black] (3) to (1);
    \draw [black] (3) to (2);
    \draw [black] (2) to (1);
\end{tikzpicture}

}
$\qquad$
\subfloat[(b)]{
\begin{tikzpicture}[scale=0.04]
   \tikzstyle {v}=[draw,circle,fill=black,inner sep=.7pt]
   \definecolor{marked}{rgb}{0.25,0.5,0.25}
    \node [v,fill] (19) at (56.198740,50.721319) {};
    \node [v,fill] (18) at (48.487936,51.481151) {};
    \node [v,fill] (17) at (47.725081,71.927222) {};
    \node [v,fill] (16) at (61.065500,58.019344) {};
    \node [v,fill] (15) at (76.840348,37.728487) {};
    \node [v,fill] (14) at (76.525070,19.133416) {};
    \node [v,fill] (13) at (65.400597,32.782136) {};
    \node [v,fill] (12) at (54.158381,36.428400) {};
    \node [v,fill] (11) at (43.832919,36.022975) {};
    \node [v,fill] (10) at (40.877466,50.237389) {};
    \node [v,fill] (9) at (35.367869,56.865634) {};
    \node [v,fill] (8) at (46.980898,94.179109) {};
    \node [v,fill] (7) at (99.999999,9.319152) {};
    \node [v,fill] (6) at (50.232054,11.885128) {};
    \node [v,fill] (5) at (49.548122,22.403506) {};
    \node [v,fill] (4) at (32.848908,31.596468) {};
    \node [v,fill] (3) at (21.051895,35.717056) {};
    \node [v,fill] (2) at (0.000000,5.820890) {};
    \node [v,fill] (1) at (22.890417,17.209705) {};
    \draw [black] (19) to (12);
    \draw [black] (19) to (18);
    \draw [black] (19) to (17);
    \draw [black] (19) to (16);
    \draw [black] (19) to (13);
    \draw [black] (18) to (10);
    \draw [black] (18) to (17);
    \draw [black] (18) to (12);
    \draw [black] (18) to (11);
    \draw [black] (17) to (8);
    \draw [black] (17) to (16);
    \draw [black] (17) to (10);
    \draw [black] (17) to (9);
    \draw [black] (16) to (8);
    \draw [black] (16) to (15);
    \draw [black] (16) to (13);
    \draw [black] (15) to (7);
    \draw [black] (15) to (14);
    \draw [black] (15) to (13);
    \draw [black] (15) to (8);
    \draw [black] (14) to (5);
    \draw [black] (14) to (13);
    \draw [black] (14) to (7);
    \draw [black] (14) to (6);
    \draw [black] (13) to (5);
    \draw [black] (13) to (12);
    \draw [black] (12) to (5);
    \draw [black] (12) to (11);
    \draw [black] (11) to (4);
    \draw [black] (11) to (10);
    \draw [black] (11) to (5);
    \draw [black] (10) to (4);
    \draw [black] (10) to (9);
    \draw [black] (9) to (3);
    \draw [black] (9) to (8);
    \draw [black] (9) to (4);
    \draw [black] (8) to (2);
    \draw [black] (8) to (7);
    \draw [black] (8) to (3);
    \draw [black] (7) to (2);
    \draw [black] (7) to (6);
    \draw [black] (6) to (1);
    \draw [black] (6) to (5);
    \draw [black] (6) to (2);
    \draw [black] (5) to (1);
    \draw [black] (5) to (4);
    \draw [black] (4) to (1);
    \draw [black] (4) to (3);
    \draw [black] (3) to (1);
    \draw [black] (3) to (2);
    \draw [black] (2) to (1);
\end{tikzpicture}
}
\end{center}
\caption{Acute spherical triangulations with 28 and 34 faces.}\label{fig:face}
\end{figure}
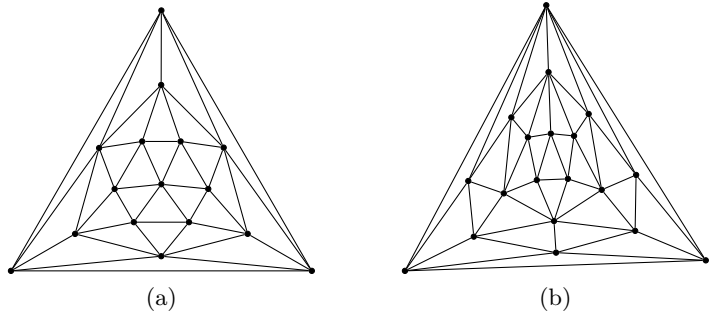

As for the theme of fixed combinatorics, Maehara determined exactly when a given abstract triangulation of a polygon can be realized as an acute triangulation in $\EE$.
Given a triangulation $L$ of a disk, a cycle $C$ in $L^{(1)}$ is said to be \emph{enclosing} if $C$ bounds
a disk with least one interior vertex.

\begin{theorem}[\cite{Maehara2003}]\label{thm:maehara}
An abstract triangulation $L$ of a disk is acute in $\mathbb{E}^2$ if and only if $L$ does not have an enclosing 3- or 4-cycle.
\end{theorem}

The forward direction of Theorem~\ref{thm:maehara} is a straightforward observation from Euclidean geometry. Maehara proved the backward direction by finding a certain nice embedding of $L$ into $\mathbb{E}^2$. Namely, he used a limit argument to construct a circle $C_v$ centered at each vertex $v\in L$ so that $C_u$ and $C_v$ are orthogonal if and only if $\{u,v\}$ is an edge of $L$. It is then elementary to see that each face of $L$ is acute in $\mathbb{E}^2$. He hinted that this construction can also be deduced from results on orthogonal circle packings on $\mathbb{S}^2$ and we will explore this approach in Section~\ref{s:planar}.

\iffalse
Regarding acute triangulations of a sphere $\mathbb{S}^n$, relatively few results are known.
It was noted in~\cite{KPP2012} that spherical analogues of results on acute triangulations of Euclidean spaces can be ``particularly insightful'' in relation to dimension reduction arguments for Euclidean triangulations.
\fi
\section{Preliminaries} \label{s:prelim}
A subset of $\mathbb{S}^2$ is called \emph{proper} if it is contained in a hemisphere.
By a \emph{spherical triangle}, we mean a triangle $R$ in $\mathbb{S}^2$ with geodesic edges such that the angles and the edges are in $(0,\pi)$. 
When there is no danger of confusion, the proper closed region bounded by $R$ is also called a \emph{spherical triangle} (or a \emph{face}, if $R$ belongs to a triangulation). 
By convention, the angles of a spherical triangle $R=ABC$ are denoted as $A,B,C$ and the lengths of their opposite edges are denoted as $a,b,c$ respectively.
The \emph{polar dual} of $R=ABC$ is defined as the spherical triangle $R'=A'B'C'$ with angles
$A' = \pi-a, B'=\pi-b,C'=\pi-c$ and with the lengths of their opposite edges $\pi - A, \pi-B$ and $\pi-C$.
Concretely, the vertices of $R'$ are suitably chosen poles of the geodesics defining $R$~\cite{Hodgsonfirstauthor}.

\subsection{Hyperbolic polyhedra}\label{s:polyhedron}
Unless specified otherwise, a polyhedron in this paper will be assumed compact, convex, 3-dimensional and \emph{simple}; that is, the valence of each vertex is three. The exception for this convention is Section~\ref{s:planar}, where non-compact, non-simple polyhedra will also be considered.

A \emph{combinatorial 2-complex} is a 2-dimensional CW complex 
where
(i) the characteristic map of each closed cell is injective and 
(ii) the boundary of each 2-face has a polygon structure such that the gluing map defines a combinatorial isomorphism.
We note that the condition (i) is often not included in the literature~\cite{BH1999}.
An \emph{abstract polyhedron} is a topological 3-ball whose boundary is equipped with the structure of a combinatorial 2-complex~\cite{Roeder}.
% $X$ homeomorphic to $\mathbb{S}^2$ such that $X^{(1)}$ is trivalent
We state Andreev's theorem~\cite{Andreev1970, Roeder} in the following form; see~\cite[Theorem 6.10.2]{Davisbook}.

\begin{theorem} (Andreev \cite{Andreev1970, Roeder})\label{thm:andreev}
Suppose $P$ is an abstract simple polyhedron, different from a tetrahedron.  Let $E$ be the edge set of $P$ and $\theta: E \rightarrow (0, \pi/2]$ be a function.  Then $(P, \theta)$ can be realized as a polyhedron in $\mathbb{H}^3$ with dihedral angles as prescribed by $\theta$ if and only if the following conditions hold: 
\begin{enumerate} [(1)]
\item At each vertex, the angles at the three edges $e_1,e_2$, and $e_3$ which meet there satisfy $\theta(e_1) + \theta(e_2) + \theta(e_3) > \pi$. 
\item If three faces intersect pairwise but do not have a common vertex, then the angles at the three edges of intersection satisfy  $\theta(e_1) + \theta(e_2) + \theta(e_3) < \pi$.
\item Four faces cannot intersect cyclically with all four angles = $\pi/2$ unless two of the opposite faces also intersect. 
\item If $P$ is a triangular prism, then the angles along the base and the top cannot all be $\pi/2$. 
\end{enumerate} 
When $(P, \theta)$ is realizable, it is unique up to isometry of $\mathbb{H}^3$. 
\end{theorem} 

A polygon or a polyhedron is \emph{all-right} if each dihedral angle is $\pi/2$.
For a combinatorial 2-complex $Y\approx\mathbb{S}^2$, we define the \emph{combinatorial dual} $Y^*$ as another combinatorial 2-complex homeomorphic to $\mathbb{S}^2$ obtained by placing a vertex in each of the 2-faces in $Y$ and joining two vertices of $Y^*$ if the corresponding faces in $Y$ are adjacent.
The following special case of Andreev's theorem is often useful.

\begin{corollary}[\cite{Andreev1970, Roeder}]\label{cor:allright}
Let $L$ be an abstract triangulation of $\mathbb{S}^2$.
The combinatorial dual $L^*$ can be realized as the boundary of an all-right hyperbolic polyhedron if and only if $L^{(1)}$ is flag no-square.
\end{corollary}

%A 3-dimensional polyhedron is \emph{simple} if each vertex has valence 3. 
The \emph{Gauss map} $G$ of an Euclidean polyhedron $P$ is a set-valued function that takes a point $x \in \partial P $ to the set of unit normals of supporting hyperplanes at $x$.  
Note that the \emph{Gauss image} $G(P)$ is combinatorially dual to $\partial P$
and isometric to $\mathbb{S}^2$.
The Gauss image of an edge $e$ of $P$ containing a vertex $v$ is a spherical arc with length the exterior dihedral angle between the two faces sharing $e$. The angles between the Gauss images of two edges $e,e'$ which meet at $v$ is the angle between two planes perpendicular to $e,e'$, respectively, oriented toward $v$. So we have the following.

\begin{observation}[\cite{Hodgsonfirstauthor}]\label{o:dual} 
For a vertex $v$ of an Euclidean polyhedron $P$, the Gauss image of $v$ 
is the polar dual of the link of $v$.   \end{observation} 

%Indeed, the link is a spherical triangle $T$. The angles of $T$ are the dihedral angles between faces of $P$ which contain $v$, and the side lengths are the corresponding face angles.  
% This is the complementary angle to the face angle. Thus the image of $v$ under the Gauss map is a spherical polygon which has edgelengths and angles, respectively, which are complementary to the angles and edgelengths, respectively,  of the link of $v$. 

Rivin and Hodgson define a similar map for a hyperbolic polyhedron.  Namely, given a hyperbolic polyhedron $P$, let $G(P)$ be the spherical complex consisting of spherical triangles which are the polar duals of the links of vertices of $P$.  Those polar duals are glued together by isometries of their faces.  This can also be described in terms of unit normals to supporting hyperplanes in the projective model, see \cite{Hodgsonfirstauthor}. Unlike the Gauss image of a Euclidean polyhedron, the Gauss image of a hyperbolic polyhedron is not isometric to $\mathbb{S}^2$.  

\begin{theorem}\cite[Theorem 1.1]{Hodgsonfirstauthor}\label{thm:HR}
A metric space $(M,g)$ homeomorphic to $\mathbb{S}^2$ can arise as the Gaussian image $G(P)$ of a compact  polyhedron $P$ in $\mathbb{H}^3$ if and only if the following conditions hold: 
\begin{enumerate}[(1)]
\item The metric $g$ has constant curvature 1 away from a finite collection of cone points $c_1,c_2,\ldots,c_i$. 
\item The cone angle at each $c_i$ is greater than $2 \pi$. 
\item The lengths of closed geodesics of $(M,g)$ are greater than $2 \pi$. 
\end{enumerate} 
Moreover, $(M,g)$ determines $P$ uniquely up to orientation preserving isometries.
\end{theorem} 
%Note that by \cite{BH1999} such a metric space is $\CAT(1)$; see Section~\ref{s:cat}.

\subsection{$\CAT(\kappa)$ spaces}\label{s:cat}
%We work with spaces which are geodesic metric spaces.
For $\kappa\in\mathbb{R}$ and $n=\mathbb{N}\cup\{0\}$, 
we let $M_\kappa^n$ be the simply connected Riemannian $n$--manifold of constant sectional curvature $\kappa$.
In particular, $M_{-1}^n\cong\mathbb{H}^n$, $M_{0}^n\cong\mathbb{E}^n$ and $M_{1}^n\cong\mathbb{S}^n$. We let $M_{\kappa}=\coprod_{n\ge0}M_\kappa^n$.

Suppose $X$ is a geodesic metric space and $\Delta$ is a geodesic triangle in $X$. 
For a real number $\kappa$, a \emph{comparison triangle in $M_\kappa^2$ for $\Delta$} is a geodesic triangle $\bar{\Delta}\subseteq M_\kappa^2$ with the same edge-lengths as $\Delta$.
Given a comparison triangle $\bar \Delta$ for $\Delta$ and a point $x$ on $\Delta$, there is a unique comparison point $\bar x$ on $\bar\Delta$ given by the isometries between the edges of $\Delta$ and $\bar\Delta$.  

\begin{definition}   
We say $X$ is $\CAT(\kappa)$ if for every geodesic triangle $\Delta$ in $X$, and every pair of points $x, y \in \Delta$, we have $d_X(x,y) \leq d_{M_\kappa^2}(\bar x, \bar y)$ for a comparison triangle $\bar\Delta$ in $M_\kappa^2$.
\end{definition} 

An \emph{$M_\kappa$-complex} $X$ is a space obtained by gluing cells in $M_\kappa$ by isometries of their faces. 
Bridson proved that if there are only finitely many isometry types of the cells used in $X$ then $X$ is a complete geodesic space with the length metric on $X$~\cite[Theorem 7.50]{BH1999}.
If $\kappa=-1,0$ or $1$, we call $X$ a \emph{piecewise hyperbolic, Euclidean} or \emph{spherical complex}, respectively.
The \emph{link} of a vertex $v$ in an $M_\kappa$-complex $X$ is the set of unit tangent vectors at $v$.
Regardless of the value of $\kappa$, the link is a piecewise spherical complex.
The following condition is called \emph{Gromov's link condition}.

\begin{theorem}[Gromov, \cite{Gromov1987}]\label{thm:gromov}
An $M_\kappa$-complex is locally $\CAT(\kappa)$ if and only if the link of each vertex is $\CAT(1)$.
\end{theorem}

The following criterion for a $\CAT(1)$ structure is an analogue of Cartan--Hadamard theorem for $\CAT(0)$ spaces~\cite{BH1999}.

\begin{lemma}\label{lem:cat1sph}
A piecewise-spherical complex $L$ is $\CAT(1)$ if and only if 
\begin{enumerate}[(i)]
\item
the link of each vertex in $L$ is $\CAT(1)$, and
\item
each closed geodesic in $L$ has length at least $2\pi$.
\end{enumerate}
\end{lemma}

For two--spheres, we will utilize a stronger condition below than just being $\CAT(1)$.

\begin{definition}\label{defn:strongly cat1}
Let $L$ be a piecewise spherical complex homeomorphic to $\mathbb{S}^2$.
We say that $L$ is \emph{strongly $\CAT(1)$} if the cone angle at each vertex is greater than $2\pi$ and each closed geodesic in $L$ is longer than $2\pi$.
\end{definition}

Strongly $\CAT(1)$ complexes are $\CAT(1)$ by Lemma \ref{lem:cat1sph}.  Since we will construct $\CAT(-1)$ complexes in Section~\ref{s:main} and~\ref{s:subordinate} and conclude that the corresponding Coxeter groups are word-hyperbolic, we record the following: 

\begin{lemma}[{\cite[Theorem 12.5.4]{Davisbook}}]\label{lem:cathyp}
If a group $G$ acts geometrically (that is, properly and cocompactly by isometries) on a $\CAT(-1)$ space then $G$ is finitely generated and word-hyperbolic.
\end{lemma}

\subsection{Right-angled Coxeter groups} \label{ss:Coxeter} 
Right-angled Coxeter groups on no-square graphs are one of the earliest examples of word-hyperbolic groups given by Gromov~\cite{Gromov1987}.

\begin{lemma}\label{lem:hyp}
Let $Y$ be a simplicial complex.
\begin{enumerate}[(1)]
\item (\cite[Lemma 8.7.2]{Davisbook})
$C(Y)$ is one-ended if and only if $Y^{(1)}$ is not a complete graph and there does not exist a complete subgraph $K$ of $Y^{(1)}$ 
such that $L\setminus K$ is disconnected.
\item (\cite[p. 123]{Gromov1987}) 
$C(Y)$ is word-hyperbolic if and only if $Y$ is no-square.
\end{enumerate}
\end{lemma}
If a complete subgraph $K$ of $Y^{(1)}$ separates $Y$
then $C(Y)$ splits as a free product over a finite group $C(K)$ and hence, has more than one end.
This implies the forward direction of (1).
Since the right-angled Coxeter group over a square contains $\Z^2$
the forward direction of (2) is obvious as well.
The backward directions of the above lemma are not immediate and we refer the readers to the literature.

In the case when $Y$ is an abstract triangulation of $\mathbb{S}^2$, the condition in the part (1) of the above lemma is satisfied if and only if $Y$ is flag. Hence Lemma~\ref{lem:hyp} implies the following.

\begin{lemma}\label{lem:sep}
Let $L$ be an abstract triangulation of $\mathbb{S}^2$.
Then $C(L)$ is one-ended and word-hyperbolic if and only if 
$L$ is flag no-square.
\end{lemma}

\section{Key ingredients of the proof} \label{s:key} 
\subsection{Cubes and spherical triangles} 

Proposition \ref{p:reflectioncube} below is an easy consequence of Andreev's theorem and will play an essential role in our proof of the main theroem.   We say that a  3-dimensional cube $P$ in $\HHH$ is a \emph{reflection invariant hyperbolic cube} (or, \emph{reflection cube} in short) if it is combinatorially isomorphic to a cube and invariant under the action of $H=\ZZ \oplus \ZZ \oplus \ZZ$ where the generators act by reflections in certain hyperplanes (called, \emph{mid-planes} of $P$).  A  fundamental domain for this action is called a \emph{right-angled slanted cube} $Q$, as shown in Figure \ref{fig:sym}.
In the figure, the vertex $O$ came from a vertex of $P$.
%When this is a right-angled slanted cube, the dihedral angles along the edges emanating from $O'$ are also right. 
Note that in this case the dihedral angles in $Q$ at the nine edges not containing $O$ are all $\pi/2$.
%See Section \ref{s:slanted} for a general \emph{slanted cube}.

\begin{proposition} \label{p:reflectioncube} Acute spherical triangles are in one-to-one correspondence with reflection invariant hyperbolic cubes.  \end{proposition} 

\begin{proof}   Let $R=ABC$ be an acute spherical triangle.
% with angles $a,b$ and $c$ all less than $\pi/2$.  
 By Theorem \ref{thm:andreev}, one can find a right-angled slanted cube $Q$ having $R$ as the link of a vertex $O$, as shown in Figure~\ref{fig:sym}. Then construct $P$ by reflecting $Q$ along the three faces that do not contain $O$. \iffalse
 Alternatively, there exists a hyperbolic polyhedron $P$ (combinatorially isomorphic to a cube) such that the four dihedral angles of $P$ intersecting each of the three (combinatorial) mid-planes are $A$, $B$ and $C$ respectively.   This labeled combinatorial cube is invariant under reflections though these mid-planes, so by the uniqueness statement in Theorem \ref{thm:andreev}, this is a reflection cube. 
 \fi

Conversely, given a reflection cube the three mid-planes cut out a hyperbolic quadrilateral with four equal angles, which must be acute. The three acute angles corresponding to the three different mid-planes are the angles of  the link of a vertex in the cube.
\end{proof} 

\begin{corollary} \label{cor:acute} The edge-lengths of an acute spherical triangle are all acute. \end{corollary}
 \begin{proof} Each face of a reflection cube is a hyperbolic quadrilateral with four equal angles, which are acute.  These face angles are the edge-lengths of the link in the cube.
\end{proof} 

\begin{corollary} \label{cor:orthocenter} 
Every acute spherical triangle has an orthocenter.
\end{corollary}
\begin{proof} 
Realize a given acute spherical triangle $R$ as the link of a vertex $O$ in a right-angled slanted cube as in Figure~\ref{fig:sym}. In the figure, the edges $OX$ and $O'X'$ lie on the same hyperplane $\Pi$ since they are both perpendicular to $\square XY'O'Z'$.
Note that $\Pi\cap \link(O)$ is a perpendicular in the triangle $R$ from a vertex to the opposite side, since $\Pi$ is perpendicular to $\square XY'O'Z'$.
Similarly we consider $\square OYO'Y', \square OZO'Z'$, to see that
$OO'\cap\link(O)$ is the orthocenter of $R$.
\end{proof} 

A right-angled slanted cube is a special case of a general \emph{slanted cube} defined as follows. Let $\mathbb{X}= \mathbb{E}^3$ or $\HHH$. 
Suppose $Q$ is a convex polyhedron in $X$ whose combinatorial structure is a cube,
with a vertex $O$.
For each edge $e$ containing $O$, we assume that $e$ is perpendicular to the face intersecting $e$ but not containing $O$. See Figure~\ref{fig:sym}.  
Then we say that $Q$ is a \emph{slanted cube} with a \emph{distinguished vertex} $O$. 
If $R$ is the link of $O$ and $R'$ is the link of the vertex $O'$ opposite to $O$, we say that $(R,R')$ is the \emph{opposite link--pair} of $Q$.
We note that the choice of a distinguished vertex is symmetric in the following sense.

\begin{lemma}\label{lem:sym}
Let $Q$ be a slanted cube in $\mathbb{X}$ and $O$ be a distinguished vertex of $Q$.
Then the vertex $O'$ opposite to $O$ satisfies the same condition as $O$;
namely, each edge $e'$ containing $O'$ is perpendicular to the face of $Q$ that intersects $e'$ but that does not contain $O'$.
\end{lemma}

\begin{proof}
See Figure~\ref{fig:sym} for a labeling of vertices in a slanted cube.
%Let $X,Y$ and $Z$ be adjacent vertices to $O$ and $X',Y'$ and $Z'$ be their opposite vertices.
Since $OX$ is perpendicular to $\square XY'O'Z'$, the faces $\square OYZ'X$ and $\square OXY'Z$ are both perpendicular to $\square XY'O'Z'$. 
Similarly, the face $\square OYZ'X$ is perpendicular to $\square YZ'O'X'$.
Hence, $O'Z' = \square XY'O'Z'\cap \square YZ'O'X'$ is perpendicular to $\square OYZ'X$.
\end{proof}

So a convex hyperbolic cube as in Figure~\ref{fig:sym} is a slanted cube if and only if the dihedral angles along the edges in bold are all right.
Polar duality can be described using a slanted Euclidean cube.
\begin{lemma}\label{lem:euc}
For spherical triangles $R$ and $R'$,
there exists a slanted Euclidean cube whose opposite link-pair is $(R,R')$ 
if and only if $R$ is the polar dual of $R'$.
\end{lemma}

\begin{proof}
To prove the forward direction, let us suppose $R=ABC$ is the link of $O$ in a slanted cube as in Figure~\ref{fig:sym}
such that $A$ intersects $OX$.
Then $A$ is the dihedral angle of $OX$ and equal to
$\angle Y' X Z'$.
Since $Y'$ and $Z'$ are the feet of the perpendicular from $O'$ to $\square OXY'Z$ and $\square OXZ'Y$ respectively, we see that $A$ is complementary to $\angle Y' O' Z'$ which is the length of an edge of the link at $O'$. The rest of the proof is straightforward by symmetry.
For the reverse direction, let $R$ lie on the unit sphere $\mathbb{S}^2\subseteq\mathbb{E}^3$ centered at a point $O$
and draw the tangent planes to $\mathbb{S}^2$ at each vertex of $R$. Let $O'$ be the intersection of these  three planes, 
and $X,Y,Z$ be the tangent points. Then we have a picture as Figure~\ref{fig:sym}.
\end{proof}

\subsection{Visual sphere decomposition}

Let $P$ be an all-right hyperbolic polyhedron.
By translating $P$ if necessary, we assume that the origin $O$ of the Poincar\'e ball model is contained in the interior of $P$. 
 We denote the set of 2-faces of $P$ by $F$. 
For each $x\in F$,  let $\gamma_x$ be the geodesic ray from $O$ which extends the perpendicular to $x$.  This naturally extends to a point on the sphere at infinity $S^2_{\infty}$.  
% if $x$ is a face and approaching $v$ if $v$ is an ideal vertex.
If $x$ and $y$ are neighboring 2-faces of $P$, then we set $\delta_{xy}$ to be the sector (often called a \emph{fan}) spanned by $\gamma_x$ and $\gamma_y$.
%We similarly define $\delta_{fv}$ if an ideal vertex $v$ belongs to a face $f$.
%Choose $\epsilon>0$ such that $P$ contains the $\epsilon$-sphere $S_\epsilon$ centered at $O$.
The \emph{visual sphere decomposition} at $O$ induced by $P$ is the cellulation $T$ of $S_\infty$ where
$T^{(0)}=S_\infty\cap (\cup_{x\in F} \gamma_x)$
and $T^{(1)}=S_\infty\cap (\cup_{x,y\in F} \delta_{xy})$.
We will regard $T$ as a cellulation of the unit sphere $\mathbb{S}^2$ with the spherical metric.  The edges of the visual sphere decomposition will be geodesic, as they are the intersection of the unit sphere with planes through the origin. 
%and call $L$ as the nerve of the faces of $P$.
%Then $L$ precisely coincides with the nerve of the weak orthogonal cap labeling determined by the faces of $P$.
Combinatorially, the set of sectors $\{\delta_{xy}\co x,y\in F\}$ partition $P$ into right-angled slanted cubes.
If we draw such a slanted cube as in Figure~\ref{fig:sym}, 
then $O'$ will correspond to a vertex of $P$
and $X,Y,Z$ will correspond to the feet of perpendiculars from $O$ to the faces of $P$.
With respect to this partitioning of $P$, we see that $T$ is the link of $P$ at $O$.
By the proof of Proposition~\ref{p:reflectioncube}, the following is immediate.

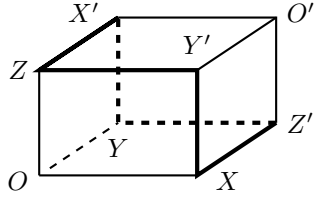
\begin{figure}[htb!]
\begin{center}
\begin{tikzpicture}[thick,scale=.7]
\draw (0,-1) node (O1) {} node [below=2pt, left] (O) {\small $O$} --(3,-1)  node (X1) {} node [below=2pt, right=3pt] (X) {\small $X$};
\draw (4.5,0)  node (ZP1) {}  node [right] (ZP) {\small $Z'$} --(4.5,2)  node [above=2pt, right] (OP) {\small $O'$} --(1.5,2)  node (XP1) {}  node [above=2pt, left=3pt] (XP) {\small $X'$};
\draw (0,1) node [left] (Z) {\small $Z$} -- (0,-1);
\draw (3,1)   node [above=2pt] (Y') {\small $Y'$} --(4.5,2);
\draw [ultra thick]  (1.5,2)--(0,1)--(3,1)--(3,-1)--(4.5,0) (1.5,2)--(0,1);
\draw [dashed] (O1) -- (1.5,0) node (Y1) {}   node [below=2pt] (Y) {\small $Y$} ;
\draw [ultra thick,dashed] (4.5,0)--(1.5,0) -- (1.5,2);
\end{tikzpicture}
\end{center}
\caption{A slanted cube. The dihedral angles in bold are right (Lemma~\ref{lem:sym}). }\label{fig:sym}
\end{figure}

\begin{lemma}\label{lem:visual}
The visual sphere decomposition $T$ obtained from an all-right hyperbolic polyhedron $P$ is acute.
\end{lemma}

%\begin{proof}Divide the polyhedron $P$ into cubes as above in the visual cube decomposition.  Each cube will be as in Figure~\ref{fig:sym}.  Since this is a right-angled polyhedron, these cubes are the fundamental domains of reflection invariant hyperbolic cubes.  Thus the link of $O$ is an acute spherical triangle by Proposition \ref{p:reflectioncube}.  The link of $O$ is exactly the spherical triangle cut out by the fans on $S_\infty$.  Thus each triangle is an acute spherical triangle. \end{proof}

We denote by $R_{2,2,2}$ the all-right spherical triangle.
Let $H=\ZZ \oplus \ZZ \oplus \ZZ$  act on the unit $\mathbb{S}^2$ by reflections through mutually perpendicular great circles
so that the underlying space of the quotient orbifold is $R_{2,2,2}$.
The \emph{$(2,2,2)$-tessellation by $R_{2,2,2}$} is defined as the universal cover of this quotient orbifold with the tiling induced by the cover.  
Concretely, the $(2,2,2)$-tessellation by $R_{2,2,2}$ is
the tiling of $\mathbb{S}^2$ cut out by the three perpendicular great circles.
%This is a metric copy of $\mathbb{S}^2$, tiled by copies of the fundamental domain.  
\begin{definition}\label{defn:222}  
Suppose $R$ is a spherical triangle.
Let us substitute $R$ for each copy of $R_{2,2,2}$ in the $(2,2,2)$-tessellation by $R_{2,2,2}$.
The resulting 2-complex $Y$ is called the  \emph{$(2,2,2)$-tessellation by $R$}.
Here we require that $Y$ is invariant under the combinatorial symmetry induced by $H$.
\end{definition}

The complex $Y$ is homeomorphic to $\mathbb{S}^2$, and the quotient orbifold by $H$ is orbifold-homeomorphic to the quotient orbifold of $\mathbb{S}^2$ by $H$.
For information on orbifolds and geometric orbifolds, we refer the reader to \cite{orbifoldbook}, \cite{frenchorbifoldbook}.  

A spherical triangle is \emph{strongly obtuse} if its dihedral angles and edge-lengths are all obtuse.  

\begin{lemma}\label{lem:222}
The $(2,2,2)$-tessellation $Y$ by a strongly obtuse spherical triangle $R$ is strongly $\CAT(1)$.
\end{lemma}

\begin{proof}
Let $R^*$ be the polar dual of $R$.
By Proposition~\ref{p:reflectioncube} there exists a reflection invariant hyperbolic cube where the link of each vertex is $R^*$. The Gauss image of this reflection invariant hyperbolic cube is $Y$. 
The proof is complete by Theorem~\ref{thm:HR}.
\end{proof}

\begin{remark}\label{rem:metric flag}
The complex $Y$ in the above lemma is a \emph{metric flag complex} as defined by Moussong.
Hence Moussong's lemma also proves that $Y$ is $\CAT(1)$~\cite[Lemma I.7.4]{Davisbook}.
However, we found the geometric approach above more generalizable for the purpose of this paper; see Lemma~\ref{lem:fat cat 22p}.
\end{remark}

\subsection{Dual Davis complex}\label{s:dual davis}
We let $L$ be an abstract triangulation of $\mathbb{S}^2$ 
and $P$ be an abstract polyhedron whose boundary is combinatorially dual to $L$.
Recall that $C(L)$ denotes the right-angled Coxeter group defined by $L$.

\begin{definition}\label{defn:dual222}
We define the \emph{dual Davis complex of $C(L)$ with respect to $P$} as the cell complex 
$X$ obtained from $P\times C(L)$ by the following gluing map:
if $g\in C(L), s\in L^{(0)}$ and $F_s$ is the face of $P$ corresponding to $s$,
then we identify $(x,g)$ to $(x,gs)$ for each point $x$ in $F_s$.
\end{definition} 
In the above definition, $C(L)$ is the group generated by reflections in the faces of the topological ball $P$, stipulating only that the reflections in neighboring faces commute.
This defines a reflection orbifold and the dual Davis complex with respect to $P$ is the universal cover of this orbifold tiled by $P$.  
%In general, the complex $X$ is a piecewise spherical, Euclidean or hyperbolic cell complex, depending on the type of $P$.  
When $C(L)$ is one-ended and word hyperbolic, there is a hyperbolic polyhedron $Q$ such that the dual Davis complex of $C(L)$ with respect to $Q$ is metrically $\HHH$ by \cite{BLP2005}.  The right-angled hyperbolic polyhedron $Q$ is the underlying space of the hyperbolic reflection orbifold.  This can also be seen using Andreev's Theorem \ref{thm:andreev}.  

%, or simply Davis' description combined with 

Let $\Gamma$ be the Cayley graph of $C(L)$ with respect to the standard generating set $L^{(0)}$.
Geometrically, we can regard the dual Davis complex $X$ as the cell complex obtained by placing a copy of $P$ at each vertex of $\Gamma$ and gluing the faces by a reflection according to the edges of the Cayley graph.
Note that $\Gamma$ is bipartite so we can place $P$ with alternating orientations and then require that each gluing map is orientation-reversing. 
It follows that $X$ is an orientable 3-manifold.

\begin{remark}\label{rem:davis}
As we are assuming $L$ is an abstract triangulation of $\mathbb{S}^2$,
we see that $X$ is homeomorphic, and combinatorially dual to, the complex $\widetilde{P_L}$ defined by Davis in~\cite[p.11]{Davisbook}.
Under the additional hypothesis that $L$ is flag, the complex $\widetilde{P_L}$ is called the \emph{Davis complex} of $C(L)$.
See also Definition~\ref{defn:dualpqr} and Lemma~\ref{lem:hypgen}.
\end{remark}
The following observation is immediate.

\begin{observation}\label{obs:link}
The link of each vertex in $X$ is the $(2,2,2)$-tessellation by the link of a vertex in $P$.
\end{observation}

\begin{lemma}\label{lem:oneend}
Let $L$ be an abstract triangulation of $\mathbb{S}^2$.
A dual Davis complex $X$ of $C(L)$ is contractible if and only if $C(L)$ is one-ended.
\end{lemma}

\begin{proof}
The complex $\widetilde{P_L}$ in Remark~\ref{rem:davis} is contractible if and only if $L$ is flag~\cite[Proposition 1.2.3]{Davisbook}. Lemma~\ref{lem:hyp} then applies.
\end{proof}

\begin{remark}\label{rem:2dim}
The $(2,2,2)$-tessellation by a spherical triangle can also be regarded as the dual Davis complex of the group $\Z_2\oplus\Z_2\oplus\Z_2$ in one lower dimension.
\end{remark}
%We remark that an almost identical definition for visual sphere decomposition is possible when $P$ is only assumed to be non-obtuse.
\section{Proof of Theorem~\ref{th:main racg}}\label{s:main}

We say a polyhedron is \emph{strongly obtuse} if all the face angles and the dihedral angles are obtuse.

\begin{lemma} \label{lem:Gaussimage} Let $T$ be an acute triangulation of $\mathbb{S}^2$.  Then there is a strongly obtuse Euclidean polyhedron $P$ such that $T$ is the Gauss image of $P$. 
\end{lemma} 

\begin{proof}
The argument is similar to the proof of Lemma~\ref{lem:euc}.
Consider the collection of the tangent planes to $\mathbb{S}^2$ at the vertices of $T$. 
The intersection of the closed half-spaces bounded by these planes and containing $\mathbb{S}^2$ is a convex Euclidean polyhedron $P$.
If we cut $P$ along the triangular open cone determined by the three vertices of an original spherical triangle $R$ we obtain a slanted cube with opposite link-pair $(R,R')$. By Lemma~\ref{lem:euc}, we see that $R'$ is the polar dual of $R$ and hence, strongly obtuse. Since $R'$ is the link of a vertex at $P$, we see that $P$ is strongly obtuse.
Note $T=G(P)$.
\end{proof}

\begin{proof}[Proof of Theorem~\ref{th:main racg}]
Suppose that $L$ is an abstract triangulation of $\mathbb{S}^2$
such that $C(L)$ is one-ended and word-hyperbolic. 
Lemma~\ref{lem:sep} implies that $L$ is flag no-square.
By Corollary~\ref{cor:allright}, there exists an all-right hyperbolic polyhedron $P$ whose boundary is combinatorially dual to $L$.
By Lemma~\ref{lem:visual}, we can realize $L$ as an acute triangulation of $\mathbb{S}^2$.

Conversely, suppose $T$ is an acute triangulation of $\mathbb{S}^2$.
By Lemma \ref{lem:Gaussimage}, there is a strongly obtuse Euclidean polyhedron $P_E$ such that $T$ is the Gauss image of $P_E$.  
Translate $P_E$ so that the origin is in the interior, and let $\lbrace x_1,x_2,\ldots,x_n \rbrace$ be the vertices of $P_E$.  Then the convex hull of the vertices ${\epsilon x_1, \epsilon x_2,\ldots, \epsilon x_n}$ is a hyperbolic polyhedron $P_H$ in $\HHH$ considered in the disc model, for a sufficiently small positive real number $\epsilon$.  The dihedral and face angles of $P_H$ are only slightly smaller than those of $P_E$. 
 In particular, we may assume by taking $\epsilon$ small enough that $P_H$ is strongly obtuse.  
We see that $\partial P_H$ is still combinatorially dual to $T$.
Let $X$ be the dual Davis complex of $C(T)$ with respect to $P_H$.
Then the link of each vertex of $X$ is $(2,2,2)$-tessellation of a strongly obtuse spherical triangle,
which is the link of a vertex in $P_H$.
By Lemma~\ref{lem:222}, the link of $X$ is strongly $\CAT(1)$.
Since $X$ is a piecewise hyperbolic cell complex, it follows that $X$ is $\CAT(-1)$.
Lemma~\ref{lem:cathyp} and Lemma~\ref{lem:oneend} imply that
$C(T)$ is word-hyperbolic and one-ended, respectively.
\end{proof}

\section{General Coxeter groups}\label{s:coxeter}
Let $Y$ be a simplicial complex and $\mathbf{m}\co Y^{(1)}\to \{2,3,4,\ldots\}$ be a labeling of the edges.
Then the \emph{Coxeter group} defined by $(Y,\mathbf{m})$ is the group presentation
\[
W(Y,\mathbf{m}) = \form{
Y^{(0)}
\;|\;
v^2=1\text{ for each vertex }v\text{ and }
(uv)^{\mathbf{m}(u,v)}=1
\text{ for each edge }
\{u,v\}
}.
\]

The right-angled Coxeter group $C(Y)$ is equal to $W(Y,\mathbf{2})$ where $\mathbf{2}$ denotes the constant function of value two. If $Y'$ is a subgraph spanned by $S\subseteq Y^{(0)}$
then $W'=\form{S}\le W$ is isomorphic to $W(Y',\mathbf{m}')$ where $\mathbf{m}'$ is the restriction of $\mathbf{m}$ to $Y'$~\cite[Theorem 4.1.6]{Davisbook}. We say $W'$ \emph{is induced by} $Y'$.

We let 
\[
W_{p,q,r}
 = \form{a,b,c\mid a^2=b^2=c^2=(bc)^p=(ca)^q=(ab)^r=1}\]
 denote the Coxeter group on a triangle such that the edge labels are $p,q$ and $r$.
Note that $W_{p,q,r}$ is finite if and only if the edge--labels $p,q,r$ satisfy the inequality $1/p+1/q+1/r > 1$. 
This is equivalent to that there exists a spherical triangle with dihedral angles $\pi/p, \pi/q$ and $\pi/r$.
Such a triangle is called a \emph{$(p,q,r)$-triangle} and denoted as $R_{p,q,r}$. 
The following is analogous to Definition~\ref{defn:222}.

\begin{definition}\label{defn:pqr}
Let $p,q,r$ be positive integers larger than 1
such that $1/p+1/q+1/r>1$.
Choose an arbitrary spherical triangle $R=ABC$
and regard the generators $a,b,c$ as the edges $BC, CA, AB$ of $R$.
Take the disjoint union of spherical triangles $R\times W_{p,q,r}$
 and identify $(x,g)$ to $(x,gs)$ whenever $g\in W_{p,q,r}$ and $x$ is on the edge of $R$ that corresponds to the generator $s$. 
 The resulting piecewise spherical complex is called the \emph{$(p,q,r)$-tessellation by $R$}.
\end{definition}

Note that the $(p,q,r)$-tessellation by $R_{p,q,r}$ is isometric to $\mathbb{S}^2$, as it is the universal cover of a spherical $(p,q,r)$ reflection orbifold. 
If $L$ is an abstract triangulation of $\mathbb{S}^2$
and $\mathbf{m}\co L^{(1)}\to\{2,3,4,\ldots\}$ is a labeling,
then we say $(L,\mathbf{m})$ is a \emph{labeled abstract triangulation} of $\mathbb{S}^2$.
We generalize the notion of dual Davis complex for Coxeter groups.

\begin{definition}\label{defn:dualpqr}
Let $(L,\mathbf{m})$ be a labeled abstract triangulation of $\mathbb{S}^2$ and put $W=W(L,\mathbf{m})$.
We assume that each face in $L$ induces a finite Coxeter group in $W$.
Take an arbitrary topological 3-ball $P$ and give $\partial P$ the dual 2-complex structure of $L$.
Each vertex $s$ of $L$ corresponds to a face $F_s$ of $P$ by the duality.
The \emph{dual Davis complex of $W$ with respect to $P$} is defined as the quotient of
 the disjoint union of polyhedra $P\times W$ by the following gluing rule:
for $g\in W, s\in L^{(0)}$ and $x\in F_s$ we identify $(x,g)$ to $(x,gs)$.
\end{definition}

As in the case of right-angled Coxeter group, 
we can regard $X$ as obtained by placing $P$ at each vertex of the Cayley graph $\Gamma$ of $W$
and gluing the faces that corresponds to the endpoints of each edge in $\Gamma$.
Then $W$ acts naturally on $X$ so that each generator ``reflects'' (namely, swaps) adjacent copies of $P$.
Each vertex $v$ of $X$ is the image of a vertex $u$ in $P$, which then corresponds to a triangle $ABC$ of $L$. Say this triangle is labeled by $p,q$ and $r$. We have the following; see Observation~\ref{obs:link}. 
\begin{observation}\label{obs:linkpqr}
In Definition~\ref{defn:dualpqr}, the link in $X$ is the $(p,q,r)$-tessellation by a link in $P$ and hence homeomorphic to $\mathbb{S}^2$.
\end{observation}
Therefore, $X$ is a manifold. 

A particularly interesting case of a dual Davis complex is when $P$ is a polyhedron in $\mathbb{S}^3,\mathbb{E}^3$ or $\HHH$. Then the complex $X$ will be a piecewise cell complex.

We record a generalized version of Lemma~\ref{lem:hyp}. 

\begin{lemma}\label{lem:hypgen}
Let $Y$ be a simplicial complex and $\mathbf{m}$ be a labeling of the edges of $Y$.
We let $W = W(Y,\mathbf{m})$.
\begin{enumerate}[(1)]
\item (\cite[p. 123]{Gromov1987}) 
$W$ is word-hyperbolic if and only if $\Z\times\Z$ does not embed into $W$.
\item (\cite[Lemma 8.7.2]{Davisbook})
$W$ is one-ended if and only if $W$ is infinite and there does not exist a complete subgraph $K$ of $Y^{(1)}$ 
such that $K$ induces a finite Coxeter group and $Y\setminus K$ is disconnected.
\item
Suppose $Y$ is an abstract triangulation of $\mathbb{S}^2$.
Then $W$ is one-ended if and only if every 3-cycle in $Y^{(1)}$ either bounds a face or induces an infinite Coxeter group. 
In this case, $X$ is contractible.
\end{enumerate}
\end{lemma}

Under the hypotheses of (3), the dual Davis complex $X$ is homeomorphic, and combinatorially dual, to the \emph{Davis complex} $\Sigma_W$ of $W$~\cite{Davisbook}, which is $\CAT(0)$. 
\section{Hyperbolic duality}\label{s:dual}
In this section we investigate \emph{hyperbolic duality} in $\HHH$ in an analogous manner to Lemma~\ref{lem:euc}.
As some of the proofs are technical in nature,
the reader may skip to the statement of Lemma~\ref{lem:fat cat 22p}
at the first reading and continue to Section~\ref{s:subordinate}.

\subsection{Hyperbolic quadrilaterals}\label{s:quadrangle}
We first need to describe the hyperbolic quadrilaterals with two opposite right-angles.
For Lemmas~\ref{lem:tri} and \ref{lem:quad}, we will consider the Poincar\'e model $\mathbb{D}$
and the Euclidean metric on $\mathbb{D}\subseteq\EE$. Let $O$ be the origin of $\mathbb{D}$.

\begin{lemma}\label{lem:tri}
Let $0<a<\pi/2$ and $A,B\in\partial \mathbb{D}$ such that $\angle AOB=a$.
We denote by $H$ the foot of the hyperbolic perpendicular from $A$ onto $\overrightarrow{OB}$.
Then $2/ (OH+1/OH) = \cos a$.
\end{lemma}
\begin{proof}
Let $O'$ denote the center of the circular arc $AH$; see Figure~\ref{fig:quad} (a).
Then 
$
OH = OO' - O'H = OO' - O'A = \sec a - \tan a$.
This implies the conclusion.
\end{proof}

\begin{figure}[htb]
\begin{center}
\subfloat[(a)]{
\begin{tikzpicture}
    \colorlet{darkbrown}{red!30!brown!70!black}    
\draw [thick] (0,0) node [left=2pt,above=2pt] {\small $O$} node (A) {} circle (1.6);
\draw [red,dashed,thick] (2,0) node (B) {} circle (1.2);
\draw (-1.28,-.96) -- (1.28,.96) node (C) {};
\draw (-1.6,0) -- (2,0);
\draw (A) node [circle,fill=blue,inner sep=2pt] {};
\draw (B) node [circle,fill=blue,inner sep=2pt] {} node [right=2pt] {\small $O'$};
\draw (C) node [circle,fill=blue,inner sep=2pt] {} node [above=5pt] {\small $A$};;
\draw (.8,0) node [circle,fill=blue,inner sep=2pt] {} node [left=6pt,below=3pt] {\small $H$};;
\draw (.5,.17) node [] {\small $a$};
\draw [darkbrown,ultra thick] (.37,0) arc (0:36.87:.37);
\end{tikzpicture}
}
$\qquad\qquad$
\subfloat[(b)]{
\begin{tikzpicture}[scale=.8]
    \colorlet{darkbrown}{red!30!brown!70!black}   
    \colorlet{darkgreen}{green!50!black}     
\draw [thick] (0,0) node [left=10pt,below=0pt] {\small $O$} node (O) {} circle (2);
\draw [red,dashed,thick] (2.5,0) node (OA) {} circle (1.5);
\draw (-1,-1.73) -- (1.1,1.9) node (OB) {};
\draw (-2,0) -- (2.5,0);
\draw (O) node [circle,fill=blue,inner sep=2pt] {};
\draw (OA) node [circle,fill=blue,inner sep=2pt] {} node [right=2pt] {\small $O_A$};
\draw (OB) node [circle,fill=blue,inner sep=2pt] {} node [right=4pt,above=1pt] {\small $O_B$};
\draw (1,0) node (A) {} node [circle,fill=blue,inner sep=2pt] {} node [left=6pt,below=3pt] {\small $A$};
\draw [darkgreen,dashed,thick] (OB) circle (1);
\draw (.5,.27) node [] {\small $c$};
\draw (3.7,-1.6) node {\small $C_A$};
\draw (2.4,2.5) node {\small $C_B$};
\draw (.85,.65) node {\small $\gamma$};
\draw (1.42,.6) node {\small $C$};
\draw (1.26,.88) node  [circle,fill=blue,inner sep=2pt] {};
\draw [darkbrown,ultra thick] (.37,0) arc (0:60:.37) ;
\draw (.598,1.035) node [circle,fill=blue,inner sep=2pt] {} node [left=2pt] {\small $B$};
\end{tikzpicture}
}
\end{center}
\caption{Lemmas~\ref{lem:tri} and~\ref{lem:quad}.}\label{fig:quad}
\end{figure}

Let us define $\xi(c,\gamma)\subseteq (0,1)^2$ to be the set of $(x,y)$ such that there exists a hyperbolic quadrilateral $OACB$ with $\angle O=c,\angle C=\gamma, \angle A= \angle B=\pi/2$ and $x = 2/(OA+1/OA), y=2/(OB+1/OB)$. See Figure~\ref{fig:quad} (b).
Recall our convention that $OA$ and $OB$ are the Euclidean lengths measured in $\mathbb{D}$.
Note that $\xi(c,\gamma)$ is non-empty only if $\gamma<\pi-c$.
Let $c\in(0,\pi)$ and $\gamma\in(0,\pi-c)$ and define $\sigma_{c,\gamma}: (0,1)^2 \rightarrow \mathbb{R}$ by 
\[
\sigma_{c,\gamma}(x,y)=\cos \gamma\sqrt{(1-x^2)(1-y^2)}-xy+\cos c.\]

\begin{lemma}\label{lem:quad}
If $c\in(0,\pi)$ and $\gamma\in(0,\pi-c)\cap (0,\pi/2]$,
then
\[\xi(c,\gamma) = \{(x,y)\;|\; \sigma_{c,\gamma}(x,y)=0\}\cap (0,1)^2.\]
\end{lemma}

\begin{proof}
Suppose $(x,y)\in\xi(c,\gamma)$ and $OACB$ is a hyperbolic quadrilateral in $\HH=\mathbb{D}$ such that $\angle O=c,\angle C=\gamma, \angle A=\angle B=\pi/2$, $x = 2 / (OA + 1/ OA)$ and $y= 2 / (OB + 1/OB)$.
Let $O_A$ be the center of the (Euclidean) circle $C_A$ containing the arc $AC$, and $O_B$ be the center of $C_B$ containing $BC$. From ${OO_A}^2= 1 + (OO_A - OA)^2$, we have $O O_A = 1/x$ and $O_A A = \sqrt{1/x^2-1}$.
One can similarly compute $O O_B$ and $O_B B$.
By the Law of Cosines, we have that ${O O_A}^2 + {O O_B}^2 - 2 O O_A\cdot O O_B \cos c = {O_A O_B}^2$ and also that ${O_A C}^2 + {O_B C}^2 - 2 O_A C\cdot O_B C\cos(\pi- \gamma)
={O_A O_B}^2 $.
Eliminating $O_A O_B$, we see that $\sigma_{c,\gamma}(x,y)=0$. 

For the converse suppose $\sigma_{c,\gamma}(x,y)=0$ for some $x,y\in(0,1)$.
As above, we let $A$ be a point so that $x = 2/(OA + 1/OA)$ and $C_A$ be the circle perpendicular to $\partial\mathbb{D}$ so that the center $O_A$ of $C_A$ is on $\overrightarrow{OA}$.
We let $B$ be such that $\angle AOB = c$ and $y = 2/(OB+1/OB)$, and similarly define $C_B$ and $O_B$.

\emph{Claim 1.} $C_A$ does not intersect $\overrightarrow{OO_B}$,
and 
$C_B$ does not intersect $\overrightarrow{OO_A}$.

By Lemma~\ref{lem:tri}, we have only to show that 
$\arccos x<c$ and $\arccos y <c$. To prove this, we simply note that
$0 = \sigma_{c,\gamma}(x,y) \ge -xy + \cos c > -x + \cos c$ and similarly, $0>-y+\cos c$.

\emph{Claim 2.} $C_A$ and $C_B$ intersect.

Again by Lemma~\ref{lem:tri}, we only need to prove that
$\arccos x+\arccos y\ge c$.
We obtain this by noting that
\[
\cos(\arccos x+\arccos y) = xy - \sqrt{(1-x^2)(1-y^2)} \le xy - \cos\gamma\sqrt{(1-x^2)(1-y^2)} =\cos c.\]

By Claim 1 and 2, we see that $x$ and $y$ determine a non-degenerate hyperbolic quadrilateral $OACB$ as shown in Figure~\ref{fig:quad} (b). In order to show the quadrilateral has the prescribed angle $\gamma$, note that 
\begin{align*}
&{O_A C}^2 + {O_B C}^2 - O_A O_B^2={O_A C}^2 + {O_B C}^2 - (OO_A^2+OO_B^2-2OO_A\cdot
OO_B\cos c)\\
&= \frac1{x^2}+\frac1{y^2}-2  - \left(\frac1{x^2}+ \frac1{y^2}\right)+\frac2{xy} \cos c
=-\frac2{xy}( xy-\cos c)\\
&=-\frac2{xy}\left(\cos\gamma\sqrt{(1-x^2)(1-y^2)}\right)
=2 O_A C\cdot O_B C\cos (\pi-\gamma).
\end{align*}
We remark that the above proof is valid without an assumption that $c$ is acute.\end{proof}

Let $c\in(0,\pi)$ and $\gamma\in(0,\pi-c)\cap (0,\pi/2]$.
The contour curve $\sigma_{c,\gamma}(x,y)=0$ is a smooth and simple arc properly embedded in $[0,1]^2$.
One way of seeing this is using the substitution $u=\cos(\arccos x+\arccos y),v=\cos(\arccos x-\arccos y)$
so that 
\[\sigma_{c,\gamma}(x,y)=0\Leftrightarrow \frac{-1-\cos \gamma}2 u + \frac{-1+\cos \gamma}2  v + \cos c=0.\]
The two cases depending on whether $c\le \pi/2$ or $c\ge\pi/2$ are drawn in Figure~\ref{fig:sigma}.
With respect to $x$, the $y$--coordinates on the curve $\sigma_{c,\gamma}=0$ are strictly decreasing.
This is because
\[
\frac{dy}{dx} = - \frac{\partial \sigma_{c,\gamma}/\partial x}{\partial \sigma_{c,\gamma}/\partial y}
= -\frac{y + x\sqrt{1-y^2}\cos\gamma / \sqrt{1-x^2}}{x +y\sqrt{1-x^2}\cos\gamma / \sqrt{1-y^2}}.\]

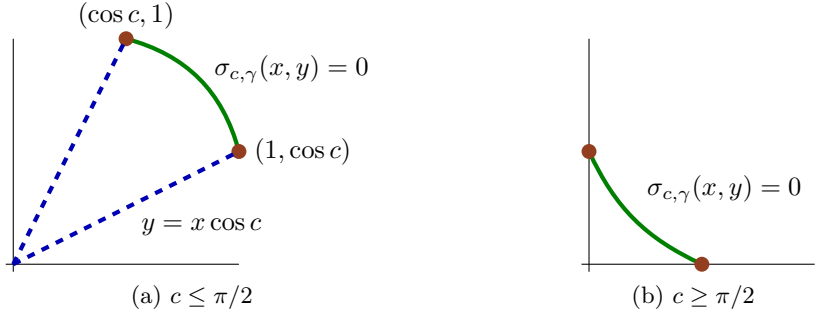
\begin{figure}[htb]
\begin{center}
\subfloat[(a) $c\le\pi/2$]
{  
\begin{tikzpicture}
    \colorlet{darkgreen}{green!50!black}
    \colorlet{darkbrown}{red!30!brown!70!black}    
    \colorlet{darkblue}{blue!70!black}
\draw (-.1,0)--(3,0);
\draw (0,-0.1)--(0,3);
\draw (1.5,3) edge [ultra thick, bend left=30,color=darkgreen] (3,1.5);
\draw (1.5,3) node [circle,inner sep=2pt, fill={darkbrown}] {};
\draw (3,1.5) node [circle,inner sep=2pt, fill={darkbrown}] {};
\draw (3.7,2.6) node [] {\small $\sigma_{c,\gamma}(x,y)=0$};
\draw (3,1.5) node [right=2pt] {\small $(1,\cos c)$};
\draw (1.5,3) node [above] {\small $(\cos c,1)$};
\draw [ultra thick,dashed,color=darkblue] (0,0) -- (3,1.5);
\draw [ultra thick,dashed,color=darkblue] (0,0) -- (1.5,3);
\draw (2.5,.5) node [] {\small $y=x\cos c$};
\end{tikzpicture}
}
$\qquad\qquad\qquad$
\subfloat[(b) $c\ge\pi/2$]
{
\begin{tikzpicture}
    \colorlet{darkgreen}{green!50!black}
    \colorlet{lightgreen}{green!80!black}
    \colorlet{darkbrown}{red!30!brown!70!black}    
    \colorlet{darkred}{red!50!black}
    \colorlet{lightred}{red!80!black}
\draw (-.1,0)--(3,0);
\draw (0,-0.1)--(0,3);
\draw (1.5,0) edge [ultra thick, bend left=20,color=darkgreen] (0,1.5);
\draw (1.5,0) node [circle,inner sep=2pt, fill={darkbrown}] {};
\draw (0,1.5) node [circle,inner sep=2pt, fill={darkbrown}] {};
\draw (1.8,1) node [] {\small $\sigma_{c,\gamma}(x,y)=0$};
\end{tikzpicture}}
\end{center}
\caption{Drawing $\sigma_{c,\gamma}(x,y)=0$.}\label{fig:sigma}.
\end{figure}

We remark that the convexity in Figure~\ref{fig:sigma} can change for different parameters $(c,\gamma)$.

\subsection{Slanted cubes revisited} \label{s:slanted}

%%%%%%%%%%%%%%%%%

Recall that the triangles in the opposite link--pair of a Euclidean slanted cube are polar dual to each other
(Lemma~\ref{lem:euc}).
Motivated by this, we define hyperbolic duality.

\begin{definition}\label{defn:hypcube}.
If two spherical triangles $R$ and $R'$ form the opposite link--pair of a slanted cube $P$ in $\HHH$, then we say that $R$ is \emph{hyperbolically dual} to $R'$ and write $P=Q(R,R')$.
\end{definition}

We have seen that hyperbolic duality is a symmetric relation (Lemma~\ref{lem:sym}).
Also, the all-right spherical triangle is hyperbolically dual to every acute spherical triangle
(Proposition~\ref{p:reflectioncube}). 
The uniqueness part of Theorem~\ref{thm:HR} implies that there exists at most one $Q(R,R')$ for a given pair of spherical triangles $R$ and $R'$.

\begin{example}
For $a\in(0,3\pi/5)$, the equilateral spherical triangle with its edge lengths $a$ is hyperbolically dual to the spherical triangle with dihedral angles $2\pi/5$.
To see this, take a regular dodecahedron $P$ in $\HHH$.
The link of a vertex in $P$ is an equilateral triangle, whose edge lengths can vary from $0$ to $3\pi/5$.
By cutting $P$ through planes passing through the center of mass of $P$ and also the midpoints of the edges of $P$, we have slanted cubes with opposite link--pair $(R,R')$ such that $R'$ is the link of a vertex in $P$ and $R=R_{2,3,5}$ is the triangle in the icosahedral tessellation of $\mathbb{S}^2$.
\end{example}

\begin{remark}\label{rem:label}
Suppose we have two spherical triangles that carry labeling of the vertices $R=ABC$ and $R'=A'B'C'$.
When we say $R$ is (polar or hyperbolically) dual to $R'$,
we will tacitly interpret this in such a way that $A$ and $A'$ correspond to opposite edges (one containing $O$ and the other containing $O'$ in Figure~\ref{fig:sym}),
and similarly for ($B,B')$ and $(C,C')$. In other words, hyperbolic duality is concerned about not only the isometry types but also the labeling of the vertices of the triangles (or equivalently a combinatorial isomorphism between $R$ and $R'$).
In particular when we say $R=ABC$ is hyperbolically dual to $R_{p,q,r}$, 
we consider the combinatorial isomorphism from $R$ to $R_{p,q,r}$ mapping $A,B$ and $C$ to the vertices of dihedral angles $\pi/p,\pi/q$ and $\pi/r$, respectively.
\end{remark}

Let $R$ be a spherical triangle and $C\subseteq \HHH$ be an open cone with the link $R$ at the vertex $O$ of $C$.
Each slanted cube where $R$ is the link of a distinguished vertex corresponds to a point $O'\in C$.
Conversely, for each $O'\in C$ let us consider the open cone $C'$ at $O'$ formed by the perpendiculars from $O'$ to the three faces of $C$. If the angles of $R$ are all non-obtuse, these three perpendiculars lie in $C$
and the intersection of $C'$ and $C$ is a slanted cube. So we have the following lemma.
The uniqueness of $O'$ (up to the symmetry of $R$) follows from that of $Q(R,R')$ where $R'=\link(O')$. Thus we have the following. 

\begin{lemma}\label{lem:cone}
Let $R$ be a spherical triangle whose dihedral angles are non-obtuse.
Then the space of spherical triangles hyperbolically dual to $R$ is parametrized by a point in the open cone whose link is $R$.\qed
\end{lemma}

When we talk about dihedral angles between two planes, we tacitly assume that each plane comes with  a transverse orientation.

\begin{lemma}\label{lem:coincidence}
Let $P_1,P_2$ and $P_3$ be three geodesic planes in $\HHH$ 
and $\alpha_i$ be the dihedral angle of $P_{i-1}$ and $P_{i+1}$ for $i=1,2,3$.
If $\alpha_1,\alpha_2,\alpha_3\in(0,\pi/2]$ and $\alpha_1+\alpha_2+\alpha_3>\pi$, then $P_1\cap P_2\cap P_3$ is non-empty.
\end{lemma}

\begin{proof}
Let us assume that  $P_1\cap P_2\cap P_3=\varnothing$.
Since $P_1,P_2,P_3$ are pairwise intersecting, at most one of $\alpha_1,\alpha_2,\alpha_3$ are $\pi/2$. So we may let $\alpha_1,\alpha_2\ne\pi/2$.
In the upper-halfspace model, we choose $P_1$ and $P_2$ to be vertical.
Let $\EE = \partial\HHH\setminus\infty$ and $Q=P_1\cap P_2\cap\EE$.
We set $C=P_3\cap\EE$ so that $Q$ is not enclosed by $C$.
We denote by $O_C$ the center of $C$.
The lines $P_1\cap\EE$ and $P_2\cap\EE$ divide $\EE$ into four open sectors.
We will first assume that the transverse orientation of $P_3$ is pointing inward of $C$.
Since $\alpha_1,\alpha_2<\pi/2$, we have that $O_C$ belongs to the open sector $S$ where the transverse orientations of $P_1$ and $P_2$ are pointing outward;
see Figure~\ref{fig:int}.
The transversely oriented lines and circles $P_1,P_2,P_3$ determine a curved triangle (shaded in the figure) in $S$ whose angles are no greater than those of a Euclidean triangle. 
Since $O_C\in S$, we see that the angles of this curved triangle are $\alpha_1,\alpha_2$ and $\alpha_3$.
  This is a contradiction to $\alpha_1+\alpha_2+\alpha_3>\pi$.
  
The case when the transverse orientation of $P_3$ points outward of $C$ is very similar.
In this case, $O_C$ must belong to the open sector $S'$ where
the transverse orientations of $P_1$ and $P_2$ are both inward. 
We again obtain a contradiction by finding a negatively curved triangle with angles $\alpha_1,\alpha_2$ and $\alpha_3$.
\end{proof}

\begin{figure}[htb]
\begin{center}
\includegraphics[height=.14\textwidth]{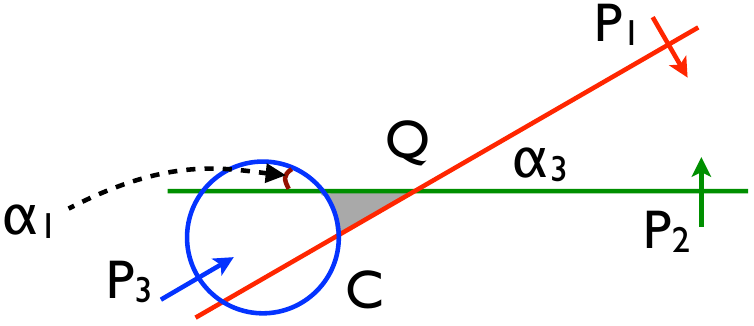}
\end{center}
\caption{Lemma~\ref{lem:coincidence}.}\label{fig:int}
\end{figure}

Let $R=ABC$ and $R'=A'B'C'$  be spherical triangles. 
We say $R$ is \emph{fatter} than $R'$ if the edge-lengths and the angles of $R$
are greater than the corresponding ones of $T'$; that is, when $A>A',BC>B'C'$ \emph{etc}.
In this case, we also say $R'$ is \emph{slimmer} than $R$.
This definition also depends on the labeling (or equivalently, a combinatorial isomorphism) of the triangles, as in the case of the hyperbolic duality.
Often, we will not mention the labeling if the underlying labeling is clear from the context.
The proof of the following is very similar to the forward direction of Lemma~\ref{lem:euc}.
\begin{lemma}\label{lem:fatgen}
Let $p,q,r\ge2$.
If a spherical triangle $R$ is hyperbolically dual to $R_{p,q,r}$ then $R$ is slimmer than the polar dual of $R_{p,q,r}$.
\end{lemma}

The converse is also true for $(2,2,p)$-triangles.

\begin{lemma}\label{lem:fat}
Suppose $p\ge2$ and $R$ is a spherical triangle.
Then $R$ is hyperbolically dual to $R_{2,2,p}$ if and only if $R$ is slimmer than the polar dual of $R_{2,2,p}$.
\end{lemma}

\begin{proof}
We only need to prove the backward direction. 
Suppose $R=ABC$ is slimmer than the polar dual of $R_{p,2,2}$.
This means that $a,A<\pi-\pi/p$ and $B,C,b,c<\pi/2$. Assume for now that the following system of equations has a solution $x,y,z\in(0,1)$: 
\[\sigma_{a,\pi/p}(y,z)=\sigma_{b,\pi/2}(z,x)=\sigma_{c,\pi/2}(x,y)=0\qquad\qquad (*)\]
We consider the Poincar\'e ball model of $\HHH=\mathbb{D}^3\subseteq\mathbb{E}^3$
and let $O$ denote the origin.
Let us pick $X,Y,Z\in\HHH$ such that $x = 2/(OX+1/OX)$,
$y = 2/(OY+1/OY)$ and
$z = 2/(OZ+1/OZ)$ where
$OX, OY$ and $OZ$ are measured in the Euclidean metric of $\mathbb{D}^3$.
By Lemma~\ref{lem:quad}, we can further require that there exists a hyperbolic quadrilateral
$\square XOYZ'$ as in Figure~\ref{fig:sym} such that $\angle XOY = c, \angle XZ'Y =\angle OXZ'=\angle OYZ' = \pi/2$.
Similarly we assume that there exist hyperbolic quadrilaterals 
$\square YOZX'$ and $\square ZOXY'$ such that $\angle YOZ = a, \angle ZOX=b, \angle YX'Z =\pi/p$
and $\angle XY'Z = \angle OYX' = \angle OZX' = \angle OZY' =\angle OXY' = \pi/2$.
Let $\Pi_X$ be the hyperplane perpendicular to $OX$ at $X$,
and define $\Pi_Y$ and $\Pi_Z$ similarly with respect to $OY$ and $OZ$.
Lemma~\ref{lem:coincidence} implies that the intersection of $\Pi_X,\Pi_Y$ and $\Pi_Z$ is non-empty, say $O'$ and we have a hyperbolic slanted cube as in Figure~\ref{fig:sym}. 
Then the link of $O'$ is $R_{2,2,p}$ as desired.

%%%%
It remains to show that $(*)$ has a solution. Note that 
$\sigma_{b,\pi/2}(z,x)=0$ is equivalent to $zx = \cos b$.
Similarly,
$\sigma_{c,\pi/2}(x,y)=0$ reduces to $xy =\cos c$.
So we have only to show that the following system has a solution $y,z\in(0,1)$:
\[\sigma_{a,\pi/p}(y,z)=0,z=y\cos b/\cos c, y>\cos c \qquad\qquad (**)\]
Note that the inequality in (**) comes from the condition $x<1$.

\emph{Case 1.} $a\in(0,\pi/2]$.

The contour curve $\sigma_{a,\pi/p}(y,z)=0$ is a smooth simple arc properly embedded in $[0,1]^2$ joining $(\cos a,1)$ and $(1,\cos a)$; see Figure~\ref{fig:sigma} (a). 
Since $B,C<\pi/2$, the spherical law of cosines for $R$ shows that
$\cos b = \cos a \cos c + \sin a \sin c \cos B > \cos a \cos c$.
Similarly we have $\cos c>\cos a \cos b$.
So $\cos a < \cos b/\cos c < 1/\cos a$ and we see that the line $z = y\cos b / \cos c$ intersects the contour curve $\sigma_{a,\pi/p}(y,z)=0$ at a unique point $(y_0,z_0)$.
Note that 
\begin{eqnarray*}
\sigma_{a,\pi/p}(\cos c,\cos b) &=& \cos(\pi/p)\sin b\sin c - \cos b \cos c +\cos a \\
&=&  \cos(\pi/p)\sin b\sin c - \cos b \cos c +  \cos b \cos c + \cos A\sin b\sin c\\
&=& ( \cos(\pi/p)+\cos A)\sin b\sin c >0
\end{eqnarray*}
This shows that $y_0>\cos c$.

\emph{Case 2.} $a\in(\pi/2,\pi)$.

The contour curve $\sigma_{a,\pi/p}(y,z)=0$ is as in Figure~\ref{fig:sigma} (b). 
It is clear that this curve intersects with $z = y\cos b/\cos c$ at a unique point $(y_0,z_0)$.
We have $\sigma_{a,\pi/p}(\cos c,\cos b)>0$ as above and so, $y_0>\cos c$.
\end{proof}

We now have an alternative statement of Proposition \ref{p:reflectioncube}.

\begin{corollary}\label{cor:acutet}
The all-right spherical triangle is hyperbolically dual to every acute spherical triangle. In other words, the space of acute spherical triangles is parametrized by the points in the first open octant in $\HHH$.
\end{corollary}

\begin{lemma}\label{lem:fat cat 22p}
If $p\ge2$
and $R$ is a spherical triangle fatter than $R_{2,2,p}$, then
%Denote by $T$ the $(2,2,p)$ spherical triangle,  and by $L$ the $(2,2,p)$ tessellation of $\SS$. Suppose $T$ is a spherical triangle fatter than $R_{2,2,p}$ with respect to a combinatorial isomorphism $\mu\co T\to T'$, and $L'$ be a spherical complex obtained by replacing $T$ in $L$ by $T'$ using $\mu$. Then $L'$ is strongly $\CAT(1)$.
the $(2,2,p)$-tessellation by $R$ is strongly $\CAT(1)$.
\end{lemma}

\begin{proof}
Let $R'$ be the polar dual of $R$.
By Lemma~\ref{lem:fat}, there is a slanted hyperbolic cube $Q$ such that the opposite link-pair is $(R_{2,2,p},R')$. We use the labeling of the vertices as in Figure~\ref{fig:sym} where $\link(O)=R_{2,2,p}$ and $\link(O')=R'$. We glue $4p$ copies of $Q$ by identifying faces containing $O$ via the reflections in the reflection group $W_{2,2,p}$ so that the vertices labeled $O$ are all identified.  Note that half of the cubes $Q$ will be oriented differently from the other half.   This union of cubes will be a hyperbolic polyhedron $P_Q$ with vertices corresponding to the vertices of the cubes labeled $O'$. The Gauss image of $P_Q$ will be the $(2,2,p)$-tessellation by $R$.  
By Theorem \ref{thm:HR}, this tessellation is strongly $\CAT(1)$.  
\end{proof}

The situation is more complicated when the spherical triangle tessellating $\mathbb{S}^2$ has more than one non-right angle. 
\begin{lemma} \label{lem:notsofat} There is a spherical triangle $R$ slimmer than the polar dual of $R_{2,3,5}$ such that $R$ is not hyperbolically dual to $R_{2,3,5}$.
\end{lemma}

\begin{proof}
Denote by $R'=A'B'C'$ the spherical triangle with $A'=\pi/2,B'=\pi/3,C'=\pi/5$ 
and set $a'=0.652\cdots,b'=0.553\cdots,c'=0.364\cdots$ to be the lengths of the opposite edges. Let $R=ABC$ be the spherical triangle such that the edge-lengths are $a=1,b=0.5$ and $c=0.6$.
%The existence of such $R$ is guaranteed by the triangle inequality.
We can compute $A = 2.318\cdots,B=0.431\cdots,C=0.514\cdots$ by spherical trigonometry 
and see that $R$ is slimmer than the polar dual of $R'$.
We claim that the three surfaces 
\[\sigma_{a,A'}(y,z)=0,\sigma_{b,B'}(z,x)=0,\sigma_{c,C'}(x,y)=0\qquad\qquad (***)\]
do not have an intersection point, as illustrated in Figure~\ref{fig:nondual} (a). 
To prove this claim, recall that the equation
$\sigma_{b,B'}(z,x)=0$ defines $z$ as a strictly decreasing function of $x$
joining $(\cos b,1)$ to $(1,\cos b)$. 
Since $\cos b = \cos 0.5 > 0.8$ we have that $\sigma_{b,B'}(z,x)=0$ only if $z$ and $x$ are both greater than 0.8.
Similarly, the relation $\cos c = \cos 0.6 >0.8$ implies that $\sigma_{c,C'}(x,y)=0$ only if $x$ and $y$ are both greater than 0.8.
But if we assume $y, z>0.8$ then
$\sigma_{a,A'}(y,z)=-yz + \cos 1 < -0.64 + 0.540\cdots < 0$.
So there are no solutions $0< x,y,z<1$ satisfying $(***)$.
Figure~\ref{fig:nondual} (b) shows the three curves $\sigma_{(a,A')}(x,y)=0, \sigma_{(b,B')}(x,y)=0$
and $\sigma_{(c,C')}(x,y)=0$, drawn in the $xy$--plane.
It follows that $R$ and $R'$ are not hyperbolically dual. 
\end{proof} 

We now show that replacing the spherical triangles in a $\CAT(1)$ complex with fatter triangles does not always produce a strongly $\CAT(1)$ complex, as might be suggested by Lemmas~\ref{lem:222} and~\ref{lem:fat cat 22p}.%The following implies that one can ``puff up'' each triangle in tessellation of $\mathbb{S}^2$ and get a spherical complex that is not strongly $\CAT(1)$.

\begin{proposition} 
 There exists a spherical triangle $R_0$ fatter than $R_{2,3,5}$ such that
 the 
 $(2,3,5)$-tessellation by $R_0$ is \emph{not} strongly $\CAT(1)$.   
\end{proposition} 
 \begin{proof} 
Let $R$ be the spherical triangle with edge lengths $a=1, b=0.5$ and $c=0.6$, and let $R_0$ be its polar dual. As above $R_0$ is fatter than $R_{2,3,5}$. 
 Assume that  the $(2,3,5)$-tessellation $T$ by $R_0$ is strongly $\CAT(1)$. 
By Theorem~\ref{thm:HR}, $T$ is the Gauss image of a hyperbolic polyhedron $P$.  Because $T$ is symmetric under the reflection group $W_{2,3,5}$  we have that $P$ is symmetric under this group, and by the Bruhat-Tits fixed point theorem, there is a point $c$ in the interior of $P$ which is fixed by this group.  Form edges emanating from $c$ which are  perpendicular to the faces of $P$.  Then consider a vertex $v$ of $P$.   The link of $v$ is the polar dual of $R_0$, namely the triangle $R$.  There are three faces of $P$ meeting $v$ and three edges emanating from $c$ which meet these faces perpendicularly.  Pairs of these edges span three planes which, together with the faces meeting $v$, form a slanted cube.  The link at $c$ of this slanted cube is $R_{2,3,5}$.  
But by Lemma \ref{lem:notsofat}, $R$ and $R_{2,3,5}$ are not hyperbolically dual so we have a contradiction.
 \end{proof} 

We say two triangles $R=ABC$ and $R'=A'B'C'$ are \emph{$\epsilon$--close} if $| A - A' |, |BC - B'C'| <\epsilon$ and similar inequalities hold for the other two pairs of angles and two pairs of edges. 
For a sufficiently small ``fat perturbation'' from the polar dual, we might always have hyperbolic duality, so we ask the following. 

\begin{figure}[t]
\begin{center}
\subfloat[(a)]{
\includegraphics[width=.3\textwidth]{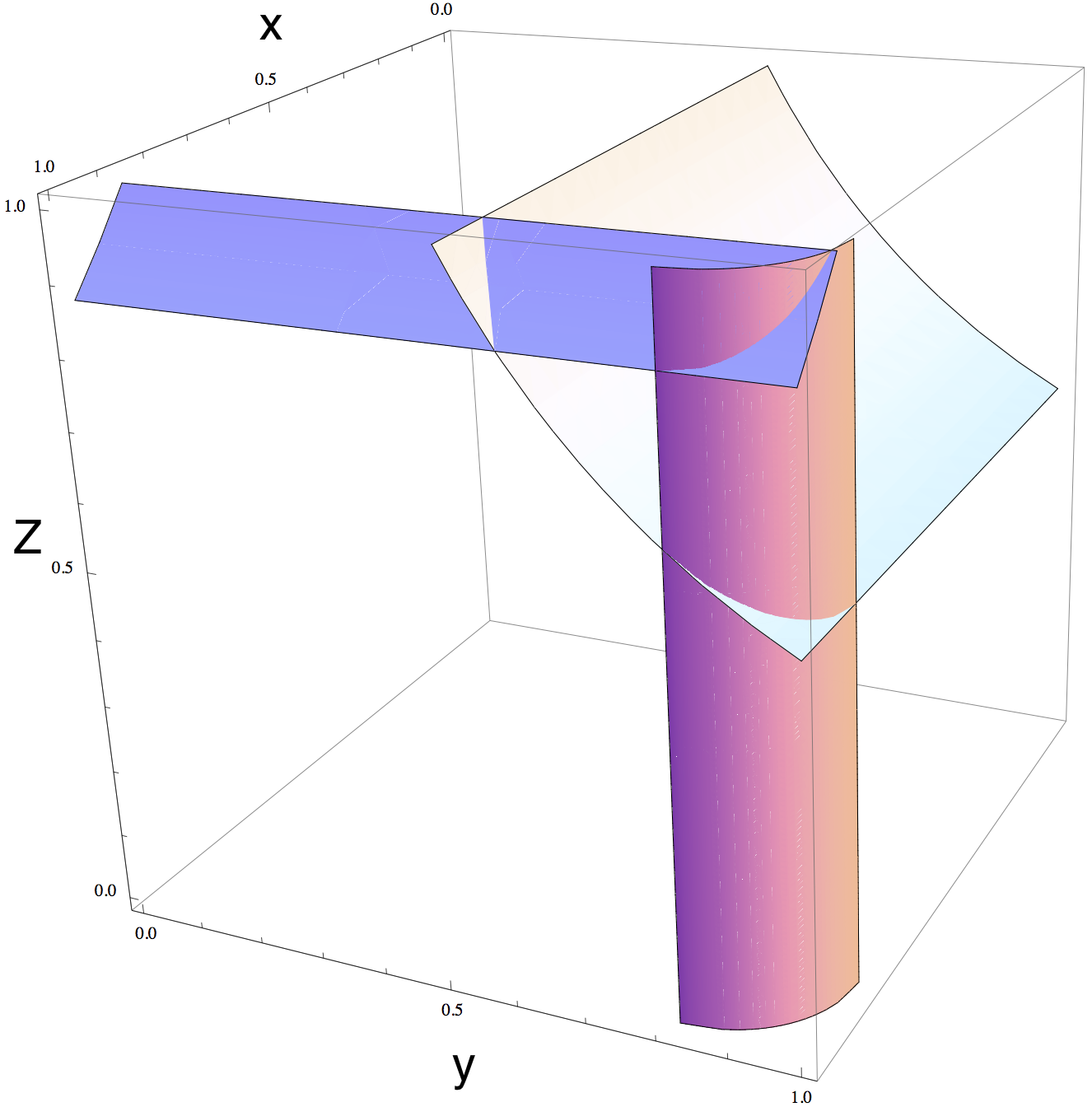}
}
$\qquad$
\subfloat[(b)]{
\includegraphics[width=.35\textwidth]{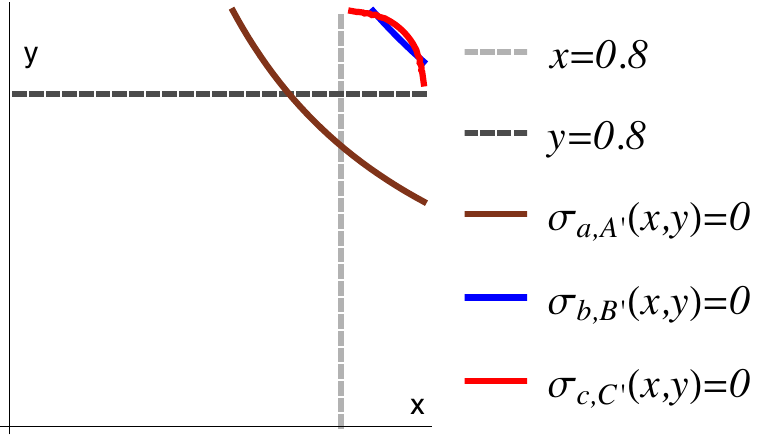}
}
\end{center}
\caption{Lemma~\ref{lem:notsofat}.}\label{fig:nondual}
\end{figure}

\begin{question}\label{que:dual}
Given a spherical triangle $R=ABC$, does there exist $\epsilon$ satisfying the following?
\begin{quote}
If $R'=A'B'C'$ is a spherical triangle which is slimmer than, and $\epsilon$--close to, the polar dual of $R$,
then $R'$ is hyperbolically dual to $R$.
\end{quote}
\end{question}

We conjecture that $\CAT(1)$-ness is preserved under sufficiently close perturbation, even in a more general situation as follows.

\begin{conjecture}[Fat CAT Conjecture]\label{conj:fat cat}
Let $T$ be a geodesic triangulations of $\mathbb{S}^2$ and denote each triangle of $T$ by $R_i = A_i B_i C_i$ for $i=1,2,\ldots,m$.
Then there exists $\epsilon>0$ satisfying the following:
\begin{quote}
Let $T'$ be a spherical complex which is combinatorially isomorphic to $T$ by a fixed combinatorial isomorphism $T\to T'$. Let us denote by $R_i'=A_i'B_i'C_i'$ the image of $R_i=A_iB_iC_i$ by this isomorphism for each $i$.
Suppose $R_i'$ is fatter than, and is $\epsilon$--close to, $R_i$ for each $i$.
Then $T'$ is $\CAT(1)$.
\end{quote}
\end{conjecture}

\section{Subordinate triangulations}\label{s:subordinate}
Throughout this section, we let $(L,\mathbf{m})$ be a labeled abstract triangulation of $\mathbb{S}^2$
such that each face induces a finite Coxeter group in $W(L,\mathbf{m})$.

\subsection{From hyperbolicity to angle constraints}\label{ss:fromhyp}
\begin{definition}\label{defn:subordinate}
Let $L'$ be a geodesic triangulation of $\mathbb{S}^2$
such that there is a combinatorial isomorphism $\mu\co L'\to L$.
Suppose that for each triangle $ABC$ of $L'$
if we let $p,q,r$ be the labels of $\mu(BC),\mu(CA)$ and $\mu(AB)$ respectively
then $ABC$ is slimmer than the polar dual of $R_{p,q,r}$.
Then we say that $L'$ is \emph{subordinate to $(L,\mathbf{m})$
(with respect to $\mu$)}.
\end{definition}

In particular, an acute triangulation is subordinate to $(L,\mathbf{2})$ by definition.
We generalize the backward direction of Theorem~\ref{th:main racg}.

\begin{theorem}\label{thm:main general}
If $W(L,\mathbf{m})$ is one-ended and word-hyperbolic, then $L$ is realized by a geodesic triangulation subordinate to $(L, \mathbf{m})$. 
\end{theorem} 
%\fi

\begin{proof} 
%Since $W(L,m)$ is one-ended, none of the spherical subgroups are separating, by Davis. Therefore the complex $L$ which is obtained from $L^{(1)} $ by adding a 2-simplex to every spherical subgraph  is exactly the triangulation $L$.   
Let $W=W(L,\mathbf{m})$ and $X$ be the dual Davis complex with respect to a topological 3-ball $P$.
%%%%%%%%%%%%%%%
%%%%%%%%%%%%%%%
%We note that $X$ is a 3-manifold and $W$ acts naturally on $X$ by reflections. A 3-manifold $M$ is irreducible if every embedded 2-sphere bounds a 3-ball, and atoroidal if every $\mathbb{Z}^2$ in $\pi_1(M)$ is peripheral and $M$ is not an $I$-bundle over the torus or Klein bottle or a disc bundle over $S^1$.  Since $X$ is irreducible and atoroidal, 
%%%%%%%%%%%%%%%
%
% How about using the following phrase, instead of the preceding paragraph?
%
Since $X$ is contractible and $W(L,\mathbf{m})$ does not contain $\mathbb{Z}^2$, the $3$--orbifold $X\big/W(L,\mathbf{m})$ can be realized as a hyperbolic orbifold by geometrization of $3$--orbifolds \cite{BLP2005}.  
The quotient reflection orbifold has the fundamental domain $P_H\approx P$, and the orbifold fundamental group is generated by reflections in the faces of $P_H$.  
In order to satisfy the relations in $W$, the link of each vertex in $P_H$ is a spherical triangle of the form $R_{p,q,r}$ for some $p,q,r$.   Move $P_H$ by an isometry so that the origin is in the interior of $P_H$.  Taking rays from the origin which meet each face perpendicularly divides $P_H$ into a union of hyperbolic slanted cubes, as in the proof of Lemma~\ref{lem:visual}.
% of partitioning of a compact hyperbolic polyhedron into cubes after Theorem \ref{thm:cap}. 
The endpoints of the rays form the vertices of a triangulation $T$ of $\mathbb{S}^2_\infty$.  
This triangulation is subordinate to $(L,\mathbf{m})$ by Lemma \ref{lem:fatgen} as each triangle on $\mathbb{S}^2_\infty$ is a link of the associated slanted cube at the origin $O$. Since $T$ is combinatorially dual to the polyhedron $P$, we have that $T$ is combinatorially isomorphic to $L$. 
\end{proof} 

\subsection{From subordinate triangulations to hyperbolicity}\label{ss:tohyp}
We prove a partial converse to Theorem~\ref{thm:main general}.

\begin{theorem}\label{thm:main 22p} 
Assume that $L$ can be realized as a geodesic triangulation that is subordinate  to $(L,\mathbf{m})$.
Suppose further that each face of $L$ has at least two edges which are labeled by $2$.
Then $W(L,\mathbf{m})$ is one-ended and word-hyperbolic.
\end{theorem}

\begin{proof}
We closely follow the proof of Theorem~\ref{th:main racg}, omitting some details.
By considering the tangent planes at the vertices of $L$,
we can find a Euclidean polyhedron $P_E$ whose Gauss image is $L$.
Set $P_H$ be a hyperbolic polyhedron whose face and dihedral angles are sufficiently close to those of $P_E$.
The link of each vertex of $P_H$ is still fatter than $R_{2,2,p}$ where we allow $p$ to be various.
We let $X$ be the dual Davis complex of $W(L,\mathbf{m})$ with respect to $P_H$.
For each vertex $v$ in $X$ there exists a vertex $u$ in $P_H$ and $p\ge2$
satisfying the following:
the link $R$ of $u$ is fatter than $R_{2,2,p}$
and the link $Y$ of $v$ is the $(2,2,p)$-tessellation by $R$.
Lemma~\ref{lem:fat cat 22p} implies that $X$ is a $\CAT(-1)$, and hence contractible, 3-manifold.
In particular, $W$ is one-ended and word-hyperbolic.
\end{proof}

\section{Planar surfaces}\label{s:planar}
In this section, we allow hyperbolic polyhedra to have ideal vertices.
\subsection{All-right hyperbolic polyhedron}\label{ss:ideal}
Let $P$ be an all-right hyperbolic polyhedron possibly with ideal vertices.
Then the combinatorial dual of $\partial P$ is a combinatorial 2-complex $L_0$ homeomorphic to $\mathbb{S}^2$
such that each face in $L_0$  is either a triangle or a square depending on whether it comes from 
a non-ideal or ideal vertex of $P$. 
For convention, we will further triangulate each square face into a \emph{square-wheel} as shown in Figure~\ref{fig:maehara} (b) and call the resulting abstract triangulation of $\mathbb{S}^2$ 
as the \emph{combinatorial nerve} of $P$.

We metrize the combinatorial nerve $L$ as follows. First we translate $P$ so that the origin of the Poincar\'{e} ball model is inside $P$.
For each vertex $u$ of $L$ corresponding to the ideal vertex $v_\infty\in\partial\HHH$ of $P$ we simply place $u$ at $v_\infty$.
If a vertex $u$ of $L$ corresponds to a face $F$ of $P$ we let $D$ be the intersection of $\partial\HHH$ and the geodesic half-spaces, not containing $P$, determined by $F$.
We then place $u$ at the center of the disk $D$.
Whenever two vertices of $L$ are adjacent we connect the corresponding vertices in $\partial\HHH$ by a geodesic. The resulting geodesic triangulation of $\partial\HHH=\mathbb{S}^2$ is called
a \emph{geometric nerve} of $P$. We will often omit the adjective ``combinatorial'' or ``geometric'' when there is no danger of confusion. The disks corresponding to non-adjacent vertices of $L$ are disjoint~\cite[Corollary 9.4]{KLV1997}.

The geometric nerve is actually a visual sphere decomposition of $P$ as defined in~Section~\ref{s:key}, slightly generalized to non-compact polyhedra. Figure~\ref{fig:quadhyp} shows an octahedron that will be cut out from $P$ in the visual sphere decomposition where $\infty$ denotes an ideal vertex of $P$ and $O$ is the origin of the Poincar\'{e} ball model. In the figure, the vertices adjacent to $O$ are the feet of the perpendiculars to the faces of $P$ which intersect the vertex $\infty$.
%These disks are called \emph{caps} in the literature.
We will need the following generalization of Corollary~\ref{cor:allright}.
For the proof, see~\cite{KLV1997}.

\begin{theorem}\cite[Theorem 9.1]{KLV1997}\label{thm:idealallright}
Let $L$ be an abstract triangulation of $\mathbb{S}^2$ which is flag.
Assume that (i) each $4$-cycle in $L^{(1)}$ bounds a region in $\mathbb{S}^2$ with exactly one interior vertex
and (ii) no pair of degree-four vertices in $L$ are adjacent.
Then there exists an all-right hyperbolic polyhedron $P$ whose combinatorial nerve is $L$.
\end{theorem}

A geometric nerve yields an acute triangulation. This is straightforward by considering slanted cubes. Or, we can prove this explicitly as follows.

\begin{lemma}\label{lem:visual2}
Let $L$ the geometric nerve of an all-right hyperbolic polyhedron possibly with ideal vertices
and $R=ABC$ be a triangle in $L$.
If $A$ corresponds to a face, then the angle of $R$ at $A$ is acute. 
If $A$ corresponds to an ideal vertex, then the angle of $R$ at $A$ is $\pi/2$.
\end{lemma}

\begin{proof}
Suppose $r,s,t$ are the radii of the disks (coming from the faces of a polyhedron $P$) at $A,B$ and $C$, and $a$, $b$, and $c$ are the lengths of the sides opposite to $A$, $B$, and $C$, respectively. 
Since the boundaries of the disks corresponding to $B$ and to $C$ are orthogonal,
the Spherical Law of Cosines implies that $\cos a = \cos s \cos t$.
Similarly we have $\cos b=\cos t\cos r$ and $\cos c = \cos r \cos s$.
So, $\cos a \ge \cos b\cos c$ and the equality holds if and only if $r=0$.
Since $\cos A = (\cos a - \cos b \cos c)/ (\sin b \sin c)$, the proof is complete.
\end{proof}

\begin{figure}[htb]
\begin{center}
\begin{tikzpicture}
    \colorlet{darkgreen}{green!50!black}
    \colorlet{lightgreen}{green!80!black}
    \colorlet{darkbrown}{red!30!brown!70!black}    
\draw [ultra thick] (5,0)-- (3.2,.3);
\draw [ultra thick,dashed] (5,0) -- (3.5,-.2);
\draw [ultra thick] (5,0) -- (3.7,.7);
\draw [ultra thick] (5,0) -- (4,-.7);
\draw [dashed,thick] (.5,0) -- (2.2,.2);
\draw [thick] (.5,0) -- (2.5,-.2);
\draw [thick] (.5,0) -- (2,.7);
\draw [thick] (.5,0) -- (2,-.7);
\draw [thick] (2,-.7)--(4,-.7)--(2.5,-.2)--(3.2,.3)--(2,.7)--(3.7,.7);
\draw [dashed,thick] (2,-.7)--(3.5,-.2)--(2.2,.2)--(3.7,.7);
\draw (5,0)  node [circle,inner sep=3pt,fill=purple] {};
\draw (.5,0)  node [circle,inner sep=3pt,fill=darkbrown] {};
\draw (3.7,.7) node [circle,inner sep=2pt,fill=red] {};
\draw (4,-.7) node [circle,inner sep=2pt,fill=red] {};
\draw (3.2,.3) node [circle,inner sep=2pt,fill=red] {};
\draw (3.5,-.2) node [circle,inner sep=2pt,fill=red] {};
\draw (2.2,.2) node [circle,inner sep=2pt,fill=blue] {};
\draw (2.5,-.2) node [circle,inner sep=2pt,fill=blue] {};
\draw (2,-.7) node [circle,inner sep=2pt,fill=blue] {};
\draw (2,.7) node [circle,inner sep=2pt,fill=blue] {};
\draw (5,0)  node [right=5pt] {\small $\infty$};
\draw (.5,0)  node [left=5pt] {\small $O$};
\end{tikzpicture}
\end{center}
\caption{Octahedron cut out by a visual sphere decomposition.}\label{fig:quadhyp}
\end{figure}
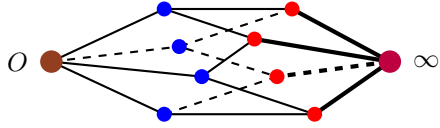

\subsection{Triangulations of planar surfaces}
Throughout this section, we let $L$ be an abstract triangulation of a planar surface.
Here,
a \emph{planar surface} means a surface obtained from $\mathbb{S}^2$ by removing finitely many (possibly zero) open disks.

A 3-cycle in $L^{(1)}$ is called \emph{empty} if it is not the boundary of a face in $L$.
Recall that an abstract triangulation $L$ is \emph{acute in $\mathbb{S}^2$ or in $\EE$} if $L$ can be realized as an acute triangulation of a subset of $\mathbb{S}^2$ or $\EE$, respectively.
%Such a triangulation $L$ is flag if and only if $L$ has no empty triangle and $L^{(1)}\ne K_4$.
%A separating triangle is empty by definition. 
The proof of the following is elementary and similar to~\cite{Maehara2003}, where Euclidean cases are considered.

\begin{lemma}\label{lem:empty}
If there is an empty 3-cycle which contains no edge in $\partial L$,
then $L$ is not acute in $\mathbb{S}^2$.
\end{lemma}

\begin{proof}
Suppose $L$ is realized as an acute triangulation in $\mathbb{S}^2$
and $C$ is an empty 3-cycle such that for each edge $e$ of $C$, there exist two faces in $L$ sharing $e$.
Choose a component $X$ of $\mathbb{S}^2\setminus C$ such that $\area(X)\le\area(\mathbb{S}^2\setminus X)$.
By Corollary~\ref{cor:acute}  the edge lengths of $C$ are less than $\pi/2$.
This implies that the area of $X$ is at most $\pi/2$,
since there exists an all-right spherical triangle containing $X$.
On the other hand, let $R_1,R_2,R_3$ be the acute triangles contained in $X$ such that each $R_i$ shares an edge with $C$.
Let $\alpha_i$ denote the angle of the vertex of $R_i$ not lying on $C$.
We let $\beta_1,\beta_2,\ldots,\beta_6$ be the interior angles of $R_1,R_2,R_3$ that are not $\alpha_1,\alpha_2$ and $\alpha_3$.
If $M$ denotes the sum of the three interior angles of $X$,
then \[M\ge\sum_{i=1}^6 \beta_i >\sum_{i=1}^3 (\pi-\alpha_i).\]
So we have
\[
\pi+\area(X)=M > \sum_{i=1}^3 (\pi-\alpha_i) >3\pi/2.\] This is a contradiction
to $\area(X)\le \pi/2$.
%Since the interior angle sum of $X$ is at least $\sum_i(\pi-\alpha_i)$, we have
%\[\pi+\area(X)=\ge\sum_i (\pi-\alpha_i) >3\pi/2$. This is a contradiction.
\end{proof}

So, an empty 3-cycle is sometimes an obstruction to acute triangulation. 
%From now on, we will only consider the case that $L$ has no empty 3-cycle in order to avoid technicality. 
We say that a cycle $C$ is $L$ is \emph{separating} if each component of $L\setminus C$ contains a vertex of $L$.
We say $L$ is \emph{flag-no-separating-square} if $L$ is flag and has no separating 4-cycle.
Note that the square-wheel is flag-no-separating-square.
% and each interior vertex of $L$ has degree at least five. This last condition is to exclude the square-wheel.

\begin{lemma}\label{lem:maehara cap}
If $L$ is flag-no-separating-square, then there exists 
a flag-no-separating-square triangulation $L'$ containing $L$ such that
$L'$ is still a planar surface and each boundary component of $L'$ is a 4-cycle.
\end{lemma}

\begin{proof}
In~\cite{Maehara2003}, Maehara considered a certain triangulation of an $n$--gon by $9n$ triangles; this triangulation is shown in Figure~\ref{fig:maehara} for the case $n=5$. Let us call this triangulation as the \emph{Maehara cap} for an $n$--gon.
For each boundary cycle of length $n\ge5$ in $L$, we glue the Maehara cap for an $n$--gon.
It is elementary to check that the resulting complex is flag-no-separating-square.
\end{proof}
\begin{figure}[t]
\begin{center}
\subfloat[(a)]{
  \tikzstyle {a}=[red,postaction=decorate,decoration={%
    markings,%
    mark=at position .5 with {\arrow[red]{stealth};}}]
  \tikzstyle {b}=[blue,postaction=decorate,decoration={%
    markings,%
    mark=at position .43 with {\arrow[blue]{stealth};},%
    mark=at position .57 with {\arrow[blue]{stealth};}}]
  \tikzstyle {v}=[draw,circle,fill=black,inner sep=1pt]
   \tikzstyle {w}=[draw,circle,fill=white,inner sep=1pt]
\begin{tikzpicture}[thick]
\foreach \i in {0,...,4} % Pentagon
    	\draw [ultra thick] (360/5*\i+90:2)--(360/5*\i+360/5+90:2) node (x\i) [v] {};
	%    	\draw [] (360/5*\i+90:.7)--(360/5*\i+360/5+90:.7) node (x\i) [v] {};
\foreach \i in {0,...,4}{
\node [v] at (360/10*2*\i+90+360/10+360/10:1.2) (y\i) {};
\node [v] at (360/10*2*\i+90+360/10:1.2) (z\i) {};
\draw (y\i)--(x\i)--(z\i);
\draw (z\i)--(360/5*\i+90:2);
\draw (y\i)--(z\i);
\draw (y\i)--(360/10*2*\i+90+3*360/10:1.2);
}

\foreach \i in {0,...,4}{
\node [v] at (360/10*2*\i+90+360/10+360/20:0.7) (Y\i) {};
\node [v] at (360/10*2*\i+90+360/20:0.7) (Z\i) {};
\draw (Y\i)--(Z\i);
\draw (Y\i)--(y\i);
\draw (z\i)--(Z\i);
\draw (Z\i)--(360/10*2*\i+90-360/10+360/10:1.2);
\draw (Y\i)--(z\i);
\draw (Y\i)--(360/10*2*\i+90+5*360/20:.7);
\node [v] at (0:0) (z) {};
\draw (Y\i)--(z)--(Z\i);
}
%\node [] at (0,0) (c) {$L$};
% \foreach \i in {0,...,4} {
%  }
\node  [] at (2.33,0) {};  % hidden point to match 2 figures
\end{tikzpicture}
}
\qquad
\subfloat[(b)]{
  \tikzstyle {a}=[red,postaction=decorate,decoration={%
    markings,%
    mark=at position .5 with {\arrow[red]{stealth};}}]
  \tikzstyle {b}=[blue,postaction=decorate,decoration={%
    markings,%
    mark=at position .43 with {\arrow[blue]{stealth};},%
    mark=at position .57 with {\arrow[blue]{stealth};}}]
  \tikzstyle {v}=[draw,circle,fill=black,inner sep=1pt]
   \tikzstyle {w}=[draw,circle,fill=white,inner sep=1pt]
\begin{tikzpicture}[thick]
\foreach \i in {0,...,3}{ % square
    	\draw [ultra thick] (360/4*\i+90+45:2)--(360/4*\i+360/4+90+45:2) node (x\i) [v] {};
	%    	\draw [] (360/5*\i+90:.7)--(360/5*\i+360/5+90:.7) node (x\i) [v] {};
	\node [v] at (0:0) (z) {};
	\draw (x\i)--(z);
}
%\node [] at (0,0) (c) {$L$};
% \foreach \i in {0,...,4} {
%  }
\node  [] at (2,0) {};  % hidden point to match 2 figures
\end{tikzpicture}
}
\end{center}
\caption{(a) Maehara cap for a pentagon. (b) Square--wheel.}\label{fig:maehara}
\end{figure}
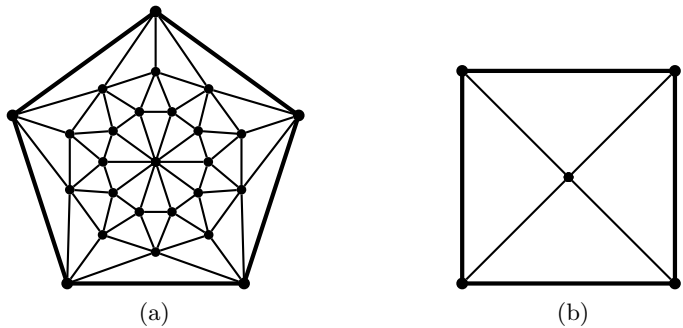

Now we prove Theorem~\ref{th:planar} (1).

\begin{theorem}\label{thm:planar1}
If $L$ is flag-no-separating-square,
then $L$ can be realized as an acute triangulation in $\mathbb{S}^2$.
\end{theorem}

\begin{proof}
By Lemma~\ref{lem:maehara cap}, we may assume that either $L\approx\mathbb{S}^2$ or $\partial L$ is a nonempty union of 4-cycles.
The former case is Theorem~\ref{th:main racg}.
For the latter, Theorem~\ref{thm:idealallright} implies that $L$ is the combinatorial nerve of an all-right hyperbolic polyhedron $P$.
The geometric nerve of $P$ is a desired acute triangulation by Lemma~\ref{lem:visual2}.
\end{proof}

Now let us consider triangulations in $\EE$.
Theorem~\ref{thm:maehara} asserts that an abstract triangulation $L$ of an $n$-gon for $n\ge5$ is acute in $\EE$ if and only if $L$ is flag-no-separating-square. As Maehara hinted in~\cite{Maehara2003}, an alternative account of this fact can be given by applying the Koebe--Andreev--Thurston theorem, which is similar to Theorem~\ref{thm:idealallright}.
We pursue this alternative approach to strengthen the result of Maehara.

We say that a geodesic triangulation $L$ in $\mathbb{S}^2$ or in $\EE$ has \emph{coinciding perpendiculars} if for each pair of neighboring faces $ABC$ and $A'BC$ of $L$, the feet of the perpendiculars from $A$ and $A'$ to $BC$ coincide. Let us give the proof of Theorem~\ref{th:planar} (2).

\begin{theorem}\label{thm:euc}
Suppose $L$ is flag-no-separating-square.
Then $L$ can be realized as an acute triangulation in $\EE$ if and only if 
at least one boundary component of $L$ is not a 4-cycle.
Furthermore, such an acute triangulation can be chosen to have coinciding perpendiculars.
\end{theorem}

\begin{proof}
Suppose $L$ is embedded in $\EE$ as a geodesic triangulation
in such a way that the outer boundary component $Q$ is a 4-cycle.
Then the four edges of $Q$ are shared by four triangles $R_1,R_2,R_3,R_4$ in $L$.
It is elementary to see that these four triangles cannot be all acute; see the proof of~\cite[Lemma 1]{Maehara2003}.

For the converse, 
%A \emph{square--wheel} is the flag complex whose $1$--skeleton is obtained from a square by adding a vertex $v$ to a square and joining $v$ to each vertex of the square.
%For each $n$--cycle in $\partial L$, we glue a Maehara cap if $n\ge5$ and a \emph{square--wheel} if $n=4$; see Figure~\ref{fig:maehara}.
 embed $L$ into some $L'$ such that $L'$ is flag-no-separating-square and each boundary component of $L'$ is a square (Lemma~\ref{lem:maehara cap}). We further glue square-wheels to each boundary component of $L'$, to get $L''$.
By Theorem~\ref{thm:idealallright}, we can realize $L''$ as the geometric nerve of an all-right hyperbolic polyhedron $P$.
For each vertex $v$ of $L''$, let us denote by $D_v\subseteq\mathbb{S}^2$ the disk corresponding to $v$ as in the definition of the geometric nerve. If $v$ has degree four, then $D_v$ is understood to be simply $v$ itself.
Since $L$ has a boundary component which is not a 4-cycle, we have glued some Maehara cap to $L$ in $L'$.
We choose a point $x_\infty$ inside the disk corresponding to the center vertex $p$ of this Maehara cap.
We mentioned that $D_p$ does not intersect $D_v$ for $v\in L^{(0)}$~\cite[Corollary 9.4]{KLV1997}.
By stereographically projecting the disks corresponding to the vertices in $L$ from the viewpoint $x_\infty$, we have orthogonally intersecting closed disks in $\EE$.
The centers of the disks in $\EE$ will form a desired acute triangulation of $L$.
Here we used an elementary fact that the centers of three pairwise orthogonally intersecting circles in $\EE$
form an acute triangle.
The ``furthermore'' part is a consequence of the following lemma, the proof of which is elementary.
\end{proof}

\begin{lemma}\label{l:euclidean perpendiculars}
Let $C_1,C_2,C_3$ be orthogonally intersecting circles in $\EE$ whose centers and radii are $P_1,P_2,P_3$ and $r_1,r_2,r_3$, respectively. Let $H$ be the foot of the perpendicular from $P_1$ to $P_2 P_3$. Then $P_2H:HP_3=r_2^2:r_3^2$.
\end{lemma}

\begin{remark}\label{rem:coinciding}
The acute triangulation of $\mathbb{S}^2$ in Theorem~\ref{th:main racg} can be chosen to satisfy coinciding perpendiculars property as well, by a similar argument.
\end{remark}

\begin{proposition}  \label{p:hemi}
There exist acute triangulations of a hemisphere. 
\end{proposition} 
\begin{proof} Take an abstract flag no-square triangulation $L$ of a disc such that the double $DL$ is also flag no-square. For example, a Maehara cap as in Figure \ref{fig:maehara} is such a triangulation.   
%Then the double $DL$ is a triangulation of a sphere without 3- or 4- cycles.  
By Corollary \ref{cor:allright}, there is a right-angled hyperbolic polyhedron $P$ such that $C(DL)$ is the Kleinian group generated by reflections in the faces of $P$.   Each outer automorphism of $C(DL)$ is induced by conjugation in $PSL(2,\mathbb{C})$  by Mostow rigidity \cite[Theorem 6.9]{frenchorbifoldbook}.  There is an outer automorphism given by the combinatorial symmetry of  $DL$ obtained by reflecting though $\partial L$. The associated isometry must take the faces on one copy of the dual of $L$ to the corresponding faces of the other, and fix the faces corresponding to vertices of $\partial L$.  Since it is the fixed point set of an isometry, the  sub-orbifold associated to the loop $\partial L$ is totally geodesic.  Move the fundamental polyhedron $P$ by an isometry of $\HHH$ so that the origin of $\HHH$ lies on a lift of this totally geodesic suborbifold.  Then the subgraph $\partial L$ of $DL$ on the sphere at infinity in the visual sphere decomposition of $P$ will be a great circle.  This realizes the triangulation of the disc as a acute triangulation of a hemisphere. 
\end{proof}

\section{Further questions} \label{s:further} 
\subsection{Two invariants of triangulations}
Let us describe two invariants $\alpha(L)$ and $\beta(T)$ of triangulations of $\mathbb{S}^2$, the first for an abstract triangulation $L$ and the second for an acute triangulation $T$.

For an abstract triangulation $L$ of $\mathbb{S}^2$ we define 
$\alpha(L)$
as the infimum of the value $\alpha$
where $L$ can be realized as a geodesic triangulation such that each dihedral angle is at most $\alpha$. For example, if $L$ is the icosahedral graph then $\alpha(L)$ is $2\pi/5$.
Corollary~\ref{cor:main} shows that $\alpha(L)$ is less than $\pi/2$ if and only if $L$ is flag-no-square.

\begin{question}\label{que:phi}
Given an abstract triangulation $L$ of $\mathbb{S}^2$, what is the combinatorial description of $\alpha(L)$?
\end{question}

It will be an interesting task to compare $\alpha(L)$ to another graph invariant, called \emph{Colin de Verdi\`ere number}~\cite{Verdiere1990,Verdiere1993}. This number reveals, among other things, whether or not a given abstract triangulation is the combinatorial nerve of an all-right hyperbolic polyhedron~\cite[Theorem 1.5]{KLV1997}. In other words, Colin de Verdi\'ere number can detect whether or not $\alpha(L)$ is less than $\pi/2$.

For an abstract triangulation $L$ of $\mathbb{S}^2$, let us denote by $\mathcal{A}(L)$ the space of acute triangulations realizing $L$.
We can naturally embed $\mathcal{A}(L)$ into $(\mathbb{S}^2)^{L^{(0)}}$ and so talk about the topology of $\mathcal{A}(L)$. Note that $\mathcal{A}(L)$ is $|L^{(0)}|$-dimensional since one may perturb a vertex of an acute triangulation and the triangulation will remain acute.   We do not know the basic facts about the topology of $\mathcal{A}(L)$.  For example 
\begin{question} \label{q:connected} Is $\mathcal{A}(L)$ connected? \end{question} 

We say that an acute triangulation $T$ in $\mathcal{A}(L)$ is \emph{geometric} if $T$ is the visual sphere decomposition of a right-angled hyperbolic polyhedron $P$ and $C(L)$ is the Kleinian group generated by reflections in the faces of $P$.   By Mostow rigidity, this visual sphere decomposition is unique up to M\"obius transformations.  Theorem \ref{th:main racg} says that when $\mathcal{A}(L)$ is non-empty, it contains a geometric triangulation.  However, our theorem does not describe how to find a geometric triangulation. 

For each acute triangle $R$ in $\mathbb{S}^2$, there uniquely exists a slanted cube $Q_R$ in $\HHH$ realizing the hyperbolic duality between $R$ and $R_{2,2,2}$. If $T$ is an acute triangulation of $\mathbb{S}^2$, we can define a \emph{volume function}:
\[
\beta(T) = \sum_{R\in T^{(2)}} \mathrm{Volume}(Q_R).\]

Then $\beta(T)$ defines an invariant of the acute triangulation $T$. This invariant might be used to explore the topological properties of $\mathcal{A}(L)$.

\begin{question}\label{que:space}
Let $L$ be an abstract triangulation of $\mathbb{S}^2$.
\begin{enumerate}[(1)]
\item
Is $\{\beta(T)\co T\in \mathcal{A}(L)\}$ an interval?
\item
%We know from Theorem~\ref{th:main racg} that $\mathcal{A}(L)\ne\varnothing$ if and only if the combinatorial dual of $L$ is realized as an all-right polyhedron $P_L\in \HHH$. 
Does $\beta(T)$ take its maximum for $T\in \mathcal{A}(L)$ if and only if $T$ is geometric? 
%\item Does $\alpha(L)$ obtain its minimum at a geometric triangulation? 
\item Is the infimum value $\alpha(L)$ attained at a geometric triangulation? 
\end{enumerate}
\end{question}

\iffalse
\subsection{General Coxeter groups}
Theorem~\ref{thm:main general} might possibly generalize to all Coxeter groups.

\begin{question}\label{que:main general}
Let $(L,\mathbf{m})$ be a labeled abstract triangulation of $\mathbb{S}^2$ such that each face of $L$ induces a finite Coxeter group in $W(L,\mathbf{m})$.
Are the following two conditions equivalent?
\begin{enumerate}[(1)]
\item
$L$ can be realized as a geodesic triangulation that is subordinate to $(L,\mathbf{m})$.
\item
$W(L,\mathbf{m})$ is one-ended and word-hyperbolic
\end{enumerate}
\end{question}
The direction (2)$\Rightarrow$(1) is proved in Theorem~\ref{thm:main general}.
The opposite direction (1)$\Rightarrow$(2) for the case that each triangle carries the label of the form $(2,2,p)$ is Theorem~\ref{thm:main 22p}.
Considering the Fat Cat Conjecture (Conjecture~\ref{conj:fat cat}) it is plausible that ``subordinate but $\epsilon$--close'' condition might be more natural in (1)$\Rightarrow$(2) direction.
\fi

\subsection{Higher-dimensional cases} \label{s:higher}

Can we classify acute triangulations of higher dimensional spheres? Let us first determine the existence.

\begin{proof}[Proof Proposition~\ref{prop:higher}]
%A spherical simplex is called \emph{proper} if it is contained in a hemisphere.
%A triangulation of a sphere consisting of proper simplices is called also \emph{proper}.
%Note that an acute simplex is proper.
%Let us prove a stronger fact that for $d\ge4$, the $d$--dimensional sphere does not admit a proper triangulation.
%Suppose $L$ is a proper triangulation of $\mathbb{S}^d$.

Suppose $L$ is an acute triangulation of $\mathbb{S}^d$, with $d\geq4$.
Since the infinite cone in $\mathbb{R}^{d+1}$ over each simplex in $L$ is convex, the convex hull of $L$ is an Euclidean polytope of dimension at least $5$, say $P\subseteq\mathbb{R}^{d+1}$. 
Kalai proved that every $n\ge 5$--dimensional Euclidean polytope has a $2$--dimensional face which is a triangle or a quadrilateral~\cite{Kalai1990}; see also~\cite[Theorem 6.11.6]{Davisbook}.
By the duality of the face lattice, this means that $P$ has a face of dimension $d-2$ that is shared by three or four faces of dimension $d-1$. Since $L$ is combinatorially isomorphic to $\partial P$, it follows that one of the dihedral angles of $L$ is at least $\pi/2$. Note that the $600$--cell is an example of an acute triangulation of $\mathbb{S}^3$.
\end{proof}
\begin{question}\label{que:higher}
Is there a combinatorial characterization for the acute triangulations of $\mathbb{S}^3$ or $\mathbb{E}^3$?
\end{question}

We believe that the question for an acute triangulation of $\mathbb{E}^3$ is particularly relevant to this paper, since the link of a vertex in such a triangulation will be an acute triangulation of $\mathbb{S}^2$. Acute triangulations of $\mathbb{S}^3$ may be related to 4-dimensional right-angled hyperbolic polytopes, which are not yet combinatorially classified. 

\section{Acknowledgements}
S. Kim is supported by Basic Science Research Program through the National Research Foundation of Korea (NRF) funded by the Ministry of Education, Science and Technology (2013R1A1A1058646). S. Kim is also supported by Samsung Science and Technology Foundation (SSTF-BA1301-06) and by Research Resettlement Fund for the new faculty of Seoul National University.
G.~S.~Walsh is supported by U. S. National Science Foundation (NSF) Grant 1207644. 
We are grateful to Andrew Casson, Ruth Charney, Andreas Holmsen, Gil Kalai, Jon McCammond and Bill Thurston for valuable and enjoyable conversations.
We especially thank Sang-il Oum for providing us a motivating example in Figure~\ref{fig:oum} (a).

\def\soft#1{\leavevmode\setbox0=\hbox{h}\dimen7=\ht0\advance \dimen7
  by-1ex\relax\if t#1\relax\rlap{\raise.6\dimen7
  \hbox{\kern.3ex\char'47}}#1\relax\else\if T#1\relax
  \rlap{\raise.5\dimen7\hbox{\kern1.3ex\char'47}}#1\relax \else\if
  d#1\relax\rlap{\raise.5\dimen7\hbox{\kern.9ex \char'47}}#1\relax\else\if
  D#1\relax\rlap{\raise.5\dimen7 \hbox{\kern1.4ex\char'47}}#1\relax\else\if
  l#1\relax \rlap{\raise.5\dimen7\hbox{\kern.4ex\char'47}}#1\relax \else\if
  L#1\relax\rlap{\raise.5\dimen7\hbox{\kern.7ex
  \char'47}}#1\relax\else\message{accent \string\soft \space #1 not
  defined!}#1\relax\fi\fi\fi\fi\fi\fi}
\providecommand{\bysame}{\leavevmode\hbox to3em{\hrulefill}\thinspace}
\providecommand{\MR}{\relax\ifhmode\unskip\space\fi MR }
\providecommand{\MRhref}[2]{%
  \href{http://www.ams.org/mathscinet-getitem?mr=#1}{#2}
}
\providecommand{\href}[2]{#2}

\end{document}